\documentclass[reqno,11pt]{amsart}
\usepackage{latexsym}
\usepackage{amsfonts}
\usepackage{amssymb}
\usepackage{mathrsfs}
\usepackage{bbm}
\usepackage{xcolor}
\usepackage{ulem}
\usepackage[mathscr]{euscript}

\usepackage[all,arc]{xy}
\usepackage{enumerate}

\usepackage [english]{babel}
\usepackage [autostyle, english = american]{csquotes}
\MakeOuterQuote{"}
\usepackage[english]{babel}
\usepackage{amsfonts}
\usepackage{mathrsfs}
\usepackage{bbm}
\usepackage{latexsym}
\usepackage{math dots}
\usepackage{amssymb}
\usepackage{mathtools}
\usepackage{relsize}
\usepackage{todonotes}
\usepackage{enumitem}

\usepackage{xpatch}
\usepackage{hyperref} 
\makeatletter
\renewcommand\th@plain{\slshape}
\xpatchcmd{\proof}{\itshape}{\slshape}{}{}
\makeatother
\makeatletter
\renewcommand\th@plain{\slshape}
\makeatother

\newcommand{\norm}[1]{\left\lVert#1\right\rVert}
\newcommand{\absval}[1]{\left\lvert#1\right\rvert}

\newcommand{\ds}{\displaystyle}
\newcommand{\R}{\mathbb{R}}
\newcommand{\Z}{\mathbb{Z}}

\newcommand{\card}{\mathrm{card}}
\newcommand{\ttt}{\mathbf{t}}

\newcommand{\ignore}[1]{}

\newcommand\hd{Hausdorff dimension}

\newcommand{\vv}{\mathbf{v}}
\newcommand{\ww}{\mathbf{w}}
\newcommand{\xx}{\mathbf{x}}
\newcommand{\yy}{\mathbf{y}}
\newcommand{\zz}{\mathbf{z}}

\newcommand{\aaa}{\mathbf{a}}

\newcommand{\les}{\preccurlyeq}

\newcommand\da{Diophantine approximation}
\newcommand\ssm{\smallsetminus}

\newcommand{\SL}{\mathrm{SL}_n(\R)}
\newcommand{\ASL}{\mathrm{ASL}_n(\R)}
\newcommand{\GL}{\mathrm{GL}_n(\R)}
\newcommand{\AGL}{\mathrm{AGL}_n(\R)}

\newcommand{\SLtwo}{\mathrm{SL}_2(\R)}
\newcommand{\Spn}{\mathrm{Sp}_{n}(\R)}

\newcommand {\comm}[1]   {\textcolor{red}{#1}}
\newcommand\eq[2]{
\begin{equation}
\label{eq:#1}
{#2}
\end{equation}
}
\newcommand{\equ}[1]{\eqref{eq:#1}}

\theoremstyle{plain}
\newtheorem{thm}{Theorem}[section]
\newtheorem{cor}[thm]{Corollary}
\newtheorem{rmk}[thm]{Remark}
\newtheorem{prop}[thm]{Proposition}
\newtheorem{lem}[thm]{Lemma}

\theoremstyle{definition}
\newtheorem{defn}[thm]{Definition}

\theoremstyle{remark}

\makeatletter
\@namedef{subjclassname@2020}{%
  \textup{2020} Mathematics Subject Classification}
\makeatother

\numberwithin{equation}{section}

\title[Khintchine-type theorems for subhomogeneous functions]{Khintchine-type theorems for values of \\ subhomogeneous functions at integer points}

\author{Dmitry Kleinbock}
\address{
\begin{itemize}
\item[] Department of Mathematics
\item[] Brandeis University 
\item[] Waltham, MA $02454$\textendash$9110$
\item[] USA
\item[] \href{mailto:kleinboc@brandeis.edu}{{\tt kleinboc@brandeis.edu}}
\end{itemize}}

\author{Mishel Skenderi}    
\address{
\begin{itemize}
\item[] Department of Mathematics
\item[] Brandeis University 
\item[] Waltham, MA $02454$\textendash$9110$
\item[] USA
\item[] \href{mailto:mskenderi@brandeis.edu}{{\tt mskenderi@brandeis.edu}} 
\end{itemize}}

\date{November 2020}
\thanks{D. K. has been supported by  NSF    grants DMS-1600814 and DMS-1900560}

\begin{document}

\begin{abstract} 

This work has been motivated by recent papers that quantify the density of values of generic quadratic forms and other  polynomials at integer points, in particular ones that use Rogers' second moment estimates.
In this paper, we establish such results in a 
very general framework. Given any subhomogeneous function (a notion to be defined) $f: \R^n \to \R$,  we derive a necessary and sufficient condition on the approximating function $\psi$ for guaranteeing that 
a generic  element $f\circ g$ in the $G$-orbit of  $f$ 
 is $\psi$-approximable; that is, 
$|f\circ g(\vv)| \le  \psi(\|\vv\|)$ for 
infinitely many $\vv \in \Z^n.$ We also deduce a sufficient condition in the case of uniform approximation. Here  $G$ can be any closed subgroup of  $\ASL$ satisfying certain axioms that allow for the use of Rogers-type estimates. 
 \end{abstract}  

\maketitle

\tableofcontents

\section{Introduction}
Let $f$ be an indefinite and nondegenerate quadratic form in $n\ge 3$ real variables that is not a real multiple of a quadratic form with rational coefficients. 
The Oppenheim\textendash Davenport Conjecture, proved in a breakthrough paper by Margulis \cite{Mar89}, states that 
$0$ is an accumulation point of $f(\Z^n)$:  in other words, for any 
$\varepsilon > 0$,
\eq{density}{\text{there exist infinitely many }\vv\in\Z^n\text{ with } |
f(\vv)| \leq \varepsilon.}
Margulis' approach, via the dynamics of unipotent flows on homogeneous spaces, was not  effective: given $
\varepsilon > 0,$ it did not give any bound on the length of the shortest integer vector $\vv$ for which \equ{density} holds. Effective versions were later established for any $n\ge 5$ \cite{BentkusGoetze, GoetzeMargulis} {using methods from analytic number theory, but these methods are not applicable to the most difficult case $n=3$.} 
One of the difficulties in establishing effective variants of Margulis' Theorem is proving the aforesaid bounds for \textit{any} quadratic form $f$ as above. This difficulty is attenuated when one seeks to prove such bounds only for \textit{generic} $f$ as above (with respect to the natural measure class).
Recently, such effective generic results have been proved both in the original setting of quadratic forms and in related settings of other homogeneous polynomials; for example, see {\cite{EMM, MM, LM, Bourgain, GhoshKelmer, Pol, GhoshGorodnikNevo, KY2, GKY_a, GKY_b,  Anishnew}}.  

\smallskip
In order to describe some of the aforementioned results in greater detail and to lay the foundation for our own work in the present paper, let us introduce some definitions. Given a norm $\nu$ on $\R^n$, a function $f: \R^n \to \R$, and a function $\psi:\R_{\geq 0} \to \R_{>0}$ (to which we shall refer as an \textsl{approximating function}), let us say that $f$ is $\left(\psi, \nu\right)$-\textsl{approximable} if 
$\varepsilon$ in the right-hand side of \equ{density} can be replaced by $ \psi\big(\nu(\xx)\big)$.
Equivalently, {$f$ is $\left(\psi, \nu\right)$-approximable if 
$\card\left(\Z^n\cap A_{f, \psi, \nu}\right) = \infty$},
 where
\eq{defapsi}{A_{f, \psi, \nu} := \left\lbrace \xx \in \R^n : |f(\xx)| \le  \psi\big(\nu(\xx)\big)\right\rbrace.}
The above definition is a way to quantify the density of $\ds f(\Z^n)$  at $0$ in terms of the approximating function $\psi$. We note that this definition 
is dependent also on the chosen norm; under some mild assumptions, however, we shall see that this is not significant for our purposes. Every specific example that we consider in this paper will satisfy these mild assumptions.
It is also clear that the definition of $(\psi, \nu)$-approximability also makes sense when $\psi$ is defined only for all sufficiently large nonnegative real numbers; however, it is convenient to assume that the domain of $\psi$ is all of $\R_{\ge 0}$ by arbitrarily extending the function, if necessary. We shall sometimes tacitly do so.

\smallskip
Consider first the special case
\begin{equation}\label{psis}
{\psi({z}) =\varphi_s({z}) := z^{-s}},
\end{equation} where {$ s \ge 0$} is arbitrary. Let $\nu$ be any norm on $\R^n.$ It was recently shown by Athreya and Margulis \cite[Theorem 1.1]{Pol} that for every $p, q \in \Z_{\geq 1}$ with $p+q = n \geq 3$,  almost every (with respect to the natural measure class) nondegenerate real quadratic form $Q$ of signature $(p, q)$ 
is $(\varphi_s, \nu)$-approximable for every $s < n-2$.
Previously this was established by {Ghosh, Gorodnik, and Nevo 
for $n = 3$ \cite{GhoshGorodnikNevo}; see also the work of  Bourgain \cite{Bourgain}, which deals with generic ternary diagonal forms. Similar results 
were 
obtained in \cite{GhoshKelmer, Pol, KY2}}.
{For instance \cite[Theorem 1]{KY2} generalizes \cite[Theorem 1.1]{Pol} as follows:} 
let  \eq{dthpowers}{\begin{aligned} f(\xx) := \sum_{j=1}^p x_j^d - \sum_{k=p+1}^n x_k^d,\\  \text{ where }p, q \in \Z_{\geq 1}, \ p+q = n  \geq 3,\text{   and } 0 < d < n &\text{ is an even integer};\end{aligned}}
then for any real $s$ with $0 < s < n-d$, almost every polynomial in the $\text{SL}_n(\R)$-orbit of $f$ is  $(\varphi_s, \nu)$-approximable.

\smallskip
We note that in all the aforementioned papers, a property stronger than $\psi$-approximability has been established. 
{Let us denote $\Z^n_{\ne0}:=  \Z^n \ssm\{{\bf 0}\}$, and} say that $f$ is \textsl{uniformly} $(\psi, \nu)$-\textsl{approximable} if for every sufficiently large $T \in \R_{>0}$, there exists {$\ds \vv \in{ \Z^n_{\ne0}}$ with \[
\nu(\vv) \le T \text{ and }|f(\vv)| \le \psi(T).\] 
{In other words, {if} for any $\varepsilon, T \in \R_{>0},$ we set
{$$ B_{f, \varepsilon, \nu, T} := \left\lbrace \xx \in \R^n : |f(\xx)| \le {\varepsilon}   \ \text{and} \ 
\nu(\xx) \le T \right\rbrace = A_{f, \varepsilon, \nu} \cap \{\xx\in\R^n: 
\nu(\xx) \le T\}$$
(here, the $\varepsilon$ in $A_{f, \varepsilon, \nu}$ stands for the constant function $\psi\equiv \varepsilon$)}, then $f$ is {uniformly} $(\psi, \nu)$-{approximable} if and only if for every sufficiently large $T \in \R_{>0},$ 
the set $B_{f, \psi(T), \nu,T}$ contains a nonzero integer vector.}}
See, for instance, \cite[\S1.1]{W}  for a discussion of asymptotic versus uniform approximation in metric number theory, and \cite{KWa, KWainh, KR} for some recent results in uniform metric \da. ("Asymptotic approximation" is the sort of approximation that we have simply called "approximation" so far in this paper.) 
It is easy to verify that if 
the approximating function $\psi$  is nonincreasing and $f$ does not represent $0$ nontrivially, then the uniform $(\psi, \nu)$-approximability of $f$ implies its $(\psi, \nu)$-approximability.
All the aforementioned papers 
actually provide conditions sufficient for the uniform $(\varphi_s, \nu)$-approximability of generic elements of the $\SL$-orbit of a given polynomial. 
For instance, \cite[Theorem 1]{KY2} 
 states that for $f$ as in  \equ{dthpowers} and for any $s < n-d$, almost every polynomial in the $\text{SL}_n(\R)$-orbit of $f$ is  uniformly $(\varphi_s, \nu)$-approximable.

\smallskip

In this paper, we 
establish a 
generalization of the aforementioned results under the mild conditions on $f$ and $\psi$ 
to which we previously alluded. 
Furthermore, our methods allow us to generalize to the case of vector-valued functions with no additional effort. We now introduce these conditions, which will require \ignore{introducing} 
some more notation and terminology. Now and hereafter, we  shall denote by $n$ an arbitrary element of $\Z_{\geq 2}$ and by $\ell$ an arbitrary element of {$\Z_{\geq 1}$}. 

\begin{defn}\label{order}
We define a non-strict partial order $\les$ on $\R^{\ell}$ as follows. For any $\ds \xx = (x_1, \dots , x_{\ell}) \in \R^{\ell}$ and any $\ds \yy = (y_1, \dots , y_{\ell}) \in \R^{\ell},$ write $\xx \les \yy$ if and only if for each $j \in \{1, \dots , {\ell}\},$ one has $x_j \leq y_j.$
\end{defn}

\begin{defn}\label{basicapdefns}
Let \[f = (f_1, \dots , f_{\ell}) : \R^n \to \R^{\ell} \ \ \ \ \ \text{and} \ \ \ \ \  \psi = (\psi_1, \dots , \psi_{\ell}) : \R_{\geq 0} \to \left(\R_{>0}\right)^{{\ell}}\] be given, and let $\nu$ be an arbitrary norm on $\R^n$.
\smallskip
\begin{itemize}
    \item  We abuse notation and write $|f|$ to denote the function $\ds (|f_1|, \dots , |f_\ell| ) :\R^n \to \R^{\ell}.$
    \item We define $\ds A_{f, \psi, \nu} := \left\{ \xx \in \R^n : |f(\xx)| \les \psi\big(\nu(\xx)\big)\right\}.$
    \item For any $T \in \R_{>0}$ and any ${\pmb\varepsilon}\in \left(\R_{>0}\right)^\ell,$ we define $$B_{f, {\pmb\varepsilon}, \nu, T} := \left\{  \xx \in \R^n : |f(\xx)| \les {\pmb\varepsilon} \ \text{and} \  
    \nu(\xx) \leq T \right\}.$$ 
    \item We say that $f$ is $(\psi, \nu)$-\textsl{approximable} if $A_{f, \psi, \nu} \cap \Z^n$ has infinite cardinality.
    \item We say that $f$ is \textsl{uniformly}  $(\psi, \nu)$-\textsl{approximable} if 
    {$B_{f, \psi(T), \nu, T} \cap {\Z^n_{\ne0}}
    \neq \varnothing$ for each sufficiently large $T \in \R_{> 0}$}. 
    \item We say that  $f$ is \textsl{subhomogeneous} if 
    $f$ is Borel measurable and there exists a constant $d = d_f \in \R_{>0}$ such that for each $t \in (0, 1)$ and each $\xx \in \R^n$ one has $\ds |f(t\xx)| \les t^d |f(\xx)|.$
    \item We say that $\psi$ is \textsl{regular} if 
    $\psi$ is Borel measurable and there exist real numbers $a = a_\psi \in \R_{>1}$ and $b = b_\psi \in \R_{>0}$ such that for each ${z} \in \R_{>0}$ one has $b_\psi\psi({z}) \les \psi(a_\psi {z}).$
    \item We say that $\psi$ is \textsl{nonincreasing} if each component function of $\psi$ is nonincreasing in the usual sense.
\end{itemize}
\end{defn}

\ignore{\begin{defn}\label{simpleapdefns} Given $f: \R^n \to \R$ and $\psi:\R_{\geq 0} \to \R_{>0}$, say that
\begin{itemize}
   \item  $f$ is \textsl{subhomogeneous} if $f$ is Borel measurable and there exists a constant $d = d_f \in \R_{>0}$ such that for each $t \in (0, 1)$ and each $\xx \in \R^k,$ one has $\ds |f(t\xx)| \le t^d |f(\xx)|;$
    \item $\psi$ is \textsl{regular} if $\psi$ is Borel measurable and there exist real numbers $a = a_\psi \in \R_{>1}$ and $b = b_\psi \in \R_{>0}$ such that for each $x \in \R_{>0}$, one has $b_\psi\psi(x) \le \psi(a_\psi x).$
\end{itemize}
\end{defn}}
Note that subhomogeneity is our only assumption on $f$; in particular, $f$ need not be a polynomial or even continuous. 
See \cite[Definition 2.2]{fkms} for another instance of using the regularity assumption on {the approximating function}
in the context of  \da. 

\smallskip

Now and henceforth, we shall 
denote by $m$ the Lebesgue measure on a Euclidean space of any dimension. (The dimension will be clear from the context.)
The following is a special case of 
our main results, Theorems \ref{mainapproxthmnew} and  \ref{UniformMain}.

\begin{thm} \label{simplemainapproxthm}
Let $\eta$ and $\nu$ be arbitrary norms on $\R^n,$ let $\ds f : \R^n \to \R^\ell$ be subhomogeneous, and let $\psi : \R_{\geq 0} \to (\R_{>0})^\ell$ be regular and nonincreasing. Then 
\smallskip

\begin{itemize}
\item[ \rm (i)] If $m\left(A_{f, \psi, \eta}\right)$ 
is finite (resp., infinite),
then $f\circ g$ is $(\psi, \nu)$-approximable for Haar 
almost no (resp., almost every)
$g \in \SL$. 
\item[\rm (ii)]
Suppose that $\ds \sum_{k=1}^{\infty}\frac1{m(B_{f, \psi(2^k), \eta, 2^k})}  < \infty$; 
then $f \circ g$ is {\sl uniformly} $\ds \left(\psi, \nu \right)$-approximable for Haar almost every $g \in \SL$.  
\end{itemize}
\end{thm}    

Part (i) is consistent with many other results in \da, where the finitude versus infinitude of the volume of a certain set provides a necessary and sufficient condition for the existence of finitely versus infinitely many solutions of certain inequalities almost everywhere. That being said, it seems remarkable that so very little needs to be assumed in order to have such a result. Moreover, a byproduct of Theorem \ref{simplemainapproxthm}(i) is that, under the above assumptions on $f$ and $\psi$, the finitude versus infinitude of $m\left(A_{f, \psi, \nu}\right)$ does not depend on the choice of the norm $\nu$. This is stated explicitly in Lemma \ref{normimmmeas} below.

\smallskip
{We shall show in \S\ref{examplesvol} that Theorem \ref{simplemainapproxthm} implies the following result, a special case of Corollary \ref{quadforms} that concerns the approximability of a function that is essentially a generalized indefinite quadratic form:}

\begin{cor}\label{spequadforms}
Let $d \in  {\R_{\geq 1}}$, 
and fix any $p, q \in \Z_{\geq 1}$ with $p+q = n$. Let $\nu$ be a norm on $\R^n.$ 
Let $\ds f : \R^n \to \R$ be given by \eq{special}{f(\xx) 
 := \sum_{j=1}^p  \left|x_j\right|^d - \sum_{k=p+1}^{n}   \left|x_{k}\right|^d.}
Let $\psi : \R_{\geq 0} \to \R_{>0}$ be regular and nonincreasing.
The following then holds:
\begin{itemize}
\item[ \rm (i)] 
 If $\ds \int_1^\infty \psi({z}) {z}^{n-(d+1)} \, d{z}$ 
is finite (resp., infinite),
then $f\circ g $ is $
(\psi, \nu)$-approximable for almost no (resp., almost every)
$g \in \SL$.

\item[ \rm (ii)]
Suppose that 
$$\begin{aligned}\sum_{k=1}^{\infty} \frac1{ k\psi(2^k)} &< \infty  \ \textsl{ if } \ d = n;\\ \sum_{k=1}^{\infty} \frac1{ 2^{(n-d)k}\psi(2^k)} &< \infty \ \textsl{ if } \   {d<n} .\end{aligned}$$ Then $f \circ g$ is  uniformly $(\psi, \nu)$-approximable for almost every  $g \in \SL.$
\end{itemize} 
\end{cor}  


{Since $d\ge 1$ in \equ{special} is assumed to be arbitrary as opposed to an even integer as in \equ{dthpowers}, the above corollary   generalizes the aforementioned work of Athreya\textendash Margulis and Kelmer\textendash Yu.}  
 In particular, we can conclude that for $\nu$ and $f$ as in Corollary \ref{spequadforms} and for almost every $g\in\SL$, the function $f\circ g$ is 
 \begin{itemize}
 \item $\ds (\varphi_{n-d}, \nu)$-approximable (the critical exponent case), and
\smallskip

 \item uniformly $\ds (\psi, \nu)$-approximable, where $\ds \psi({z}) =  \frac{(\log {z})^{1+\varepsilon}}{{z}^{n-d}}$ for an arbitrary $\varepsilon > 0$ (the critical exponent case with a logarithmic correction).
  \end{itemize} 


 
 
{We note that for any regular and nonincreasing $\psi : \R_{\geq 0} \to \R_{>0}$ and any $d \in \R_{>n},$ the integral in 
{Corollary \ref{spequadforms}(i)} converges because it is majorized by $$ \psi(1) \int_1^\infty  {z}^{n-(d+1)} \ d{z} < \infty;$$ in this case, almost every element in the $\SL$ orbit of $f$ in $\equ{special}$ is not $(\psi, \nu)$-approximable and hence is not uniformly $(\psi, \nu)$-approximable. } Other applications of Theorem \ref{simplemainapproxthm}  can be found in \S\ref{examplesvol}.

\smallskip

Historically, there have been several different approaches to this circle of problems. In particular, the papers  \cite{GhoshKelmer} and \cite{GhoshGorodnikNevo} continue the line of thought behind Margulis' proof of the Oppenheim Conjecture, reducing the problem to studying the action of the stabilizer of the function $f$ on the space of lattices, and using ergodic properties of the action to establish quantitative density of $f(\Z^n)$. In the present paper, however, we follow the methods of \cite{Pol,  KY2}, which have their origin in the work of 
{Rogers and Schmidt \cite{RogSet, Schmidt} 
and involve studying the} asymptotics of the number of lattice points of generic lattices in families of subsets of $\R^n$.
One of the advantages of the approach taken in this paper is that it makes it possible to significantly generalize the setting. In particular, one can {work with} vector-valued functions $f = (f_1, \dots , f_\ell) : \R^n \to \R^\ell$, 
and can consider 
specific subsets of $\Z^n,$ for example the set of all primitive integer points $\Z^n_{\rm pr}$. 
\smallskip

It is also worth mentioning that  the aforementioned papers were dealing with the density of $f(\Z^n)$ in $\R$, not just at zero. In other words, for various examples of polynomials $f$, these papers presented conditions depending on $s \in \R_{>0}$ sufficient for showing that for every $\xi \in \R$, almost every $g \in \SL,$ 
and every sufficiently large $T \in \R_{>0}$ 
there exists $\vv 
\in \Z^n_{\ne0}$ for which 
\eq{inhsystem} {{\nu(\vv) \le T} \text{ and }|\xi - f(g\vv)| \le T^{-s}.} 
{See, for instance, the two recent papers \cite{GKY_a, GKY_b} of Ghosh\textendash Kelmer\textendash Yu.} 
{We discuss a possible approach to this case, the inhomogeneous one, in \S\ref{inh}, and plan to address it in a forthcoming paper.}

\smallskip

Theorems \ref{mainapproxthmnew} and  \ref{UniformMain}, our main results, are essentially a generalization of Theorem \ref {simplemainapproxthm}   to a class of groups that act on $\R^n$ 
and 
satisfy certain axioms, which $\SL$ happens to satisfy. Another example of such a group is $\Spn$, the group of symplectic linear isomorphisms of $\R^{n}$ when $n \in \Z_{>0}$ is 
even, or the group $\ds \ASL := \SL \ltimes \R^n$ of unimodular affine isomorphisms of $\R^n$. For the infinite measure case of Theorem \ref{mainapproxthmnew}, we actually obtain
a quantitative version 
when we stipulate that the element $g$ lie in an arbitrary fixed compactum of the group. 

\smallskip

Let us briefly delineate the structure of this paper. In \S\ref{generalities}, we define a class of groups that satisfy certain axioms conducive to proving our main {Diophantine results}.
The utility of these axioms is that they enable us to prove generic counting results in certain spaces of lattices; our approach is a generalization of the method developed by 
Schmidt in \cite{Schmidt}. Using the axioms on $f$ and $\psi$ that have already been introduced, we then proceed {in \S\ref{zerofull}} to transfer the results concerning the space of lattices to {those} concerning Diophantine approximation. In \S\ref{examplesvol},  we then 
{discuss} 
specific examples of subhomogeneous $f$ {to obtain conditions for approximability in terms of the convergence or divergence of certain infinite series or improper integrals, as in Corollary \ref{spequadforms}.}  


\ignore{Since our results yield information about multivariate Diophantine approximation, it is only fitting that we present an example with $r > 1$ in this introductory section.
\begin{cor}\label{quadlinear}
Let $\nu$ be an arbitrary norm on $\R^3.$ Define $f: \R^3 \to \R^2$ by \[f(x_1,x_2,x_3) := (x_1^2 + x_2^2 - x_3^2, x_1),\]
and let $\ds \psi = (\psi_1,\psi_2): \R_{\geq 0}\to \left(\R_{>0}\right)^2$ be regular and nonincreasing.

For almost every $g \in G,$ the function $f\circ g$ is $\ds \left(\psi, \nu, \Z_{\mathrm{pr}}^n\right)$-approximable (resp., is not $(\psi, \nu, \Z^n)$-approximable) if the integral  \[ \ds \int_1^\infty \psi_1(r)\psi_2(r) \,dr \] is infinite (resp., is finite).
\end{cor}  
\comm{Need to prove or disprove this.}

\smallskip
}

Possible examples with which we do not concern ourselves here abound: one can, for example, take $f$ to be a system of several quadratic forms or a pair consisting of a quadratic and a linear form, as in the papers \cite{G,G1, Anishnew}. 
{It also appears very likely that one could use \cite[Proposition 5.2 and Theorem 6.1]{H} to prove $S$-arithmetic analogues over $\mathbb{Q}$ of the results of this paper. Further possible extensions and open questions are mentioned in \S\ref{cr}.}

\subsection*{Acknowledgements}
The first-named author is immensely grateful to Gregory Margulis for a multitude of conversations on the subject of the Oppenheim Conjecture and related topics. Thanks are also due to  Jayadev Athreya, Anish Ghosh, Alex Gorodnik, {Jiyoung Han}, Dubi Kelmer, {Dave Morris, and Amos Nevo} for stimulating  discussions, {and to the anonymous referee for several useful suggestions}.

\section{
Counting results for generic lattices}\label{generalities}




Let $G$ be a closed subgroup of  $\ASL$, and let  
$\Gamma$ be the subgroup of $G$ 
defined by \eq{gamma}{ \Gamma
:= \{ g \in G : g\Z^n = \Z^n \}.} Now and hereafter, we assume that $\Gamma$ is a lattice in $G$; that is, $\Gamma$ is a discrete subgroup of $G$ whose covolume in $G$ is finite. (In each particular example of such a group $G$ that we shall consider, the subgroup $\Gamma$ will indeed be a lattice in $G.$) Set $X := G/\Gamma.$ Notice that we then have a well-defined bijection between 
$X$ and $\ds \{ g\Z^n : g \in G\}$ that is given by $g\Gamma \longleftrightarrow g\Z^n.$ We therefore identify $X$ with $\ds \{ g\Z^n : g \in G\}$, and we equip $X$ with the quotient topology.

Now let ${{\mathcal{P}}}$ be any $\Gamma$-invariant subset of $\Z^n.$ Given any $\Lambda \in X,$ fix any $g \in G$ for which $\Lambda = g\Z^n$; then define $\Lambda_{{\mathcal{P}}} := g{{\mathcal{P}}}.$ Then $\Lambda_{{\mathcal{P}}}$ is well-defined because ${{\mathcal{P}}}$ is $\Gamma$-invariant.

Given any 
function $f : \R^n \to \R_{\geq 0},$ we 
define its ${{\mathcal{P}}}$-\textsl{Siegel transform} $\ds \widehat{f}^{{\,}^{\mathcal{P}}} : X \to {[0, \infty]}$ by \[ \widehat{f}^{{\,}^{\mathcal{P}}}(\Lambda) := \sum_{\vv \in \Lambda_{{\mathcal{P}}}} f(\vv). \]



%

\smallskip

We equip $G$ with the left Haar measure $\mu_G$ that is normalized so that any fundamental Borel set in $G$ for $X$ has $\mu_G$-measure equal to $1.$ We then let $\mu_X$ be the left $G$-invariant Borel probability measure on $X$ that is induced from $\mu_G$ in the canonical manner. Note that if $f$ is Borel measurable, then $\widehat{f}^{{\,}^{{\mathcal{P}}}}$ is $\mu_X$-measurable.

\ignore{Given any function $f : \R^n \to  \R_{\geq 0}$, we define its \textsl{Siegel transform} $\widehat{f}: X \to \R \cup \{\infty\}$ by \[ \widehat{f}(\Lambda) := \sum_{\vv \in  \Lambda \ssm \{\textbf{0}\}} f(\vv).\] (In the right-hand side of the above definition, we make use of the aforementioned bijective correspondence),  
and its \textsl{primitive Siegel transform} $\widetilde{f}: X \to \R \cup \{\infty\}$ by \[ \widetilde{f}(\Lambda) := \sum_{\vv \in \Lambda_{\mathrm{pr}}} f(\vv).\] Here, $\Lambda_{\mathrm{pr}}$ denotes the set of all primitive points of $\Lambda.$ Note that if $f$ is Borel measurable, then each of $\widehat{f}$ and $\widetilde{f}$ is $\mu_X$-measurable. Finally, we let $\zeta$ denote the Euler-Riemann zeta function.}

\smallskip

Let us now introduce the axioms on $G$ to which we alluded at the end of the 
introduction.

\begin{defn}\label{SiegelRogers} Let $G$ and ${{\mathcal{P}}}$ be as above.
\smallskip

\item[ \rm (i)] We say that $G$ is of ${{\mathcal{P}}}$-\textsl{Siegel type} if there exists a constant $c = {c}_{{{\mathcal{P}}}} \in \R_{>0}$ such that for any bounded and compactly supported Borel measurable function $f: \R^n \to \R_{\geq 0}$ we have \eq{siegel}{ \int_{X}\widehat{f}^{{\,}^{{\mathcal{P}}}} \, d\mu_X = c \int_{\R^n}f \, dm .}
\smallskip


\item[ \rm (ii)]  
Let $\ds r \in \R_{\geq 1}$ 
be given. We say that $G$ is of $\ds \left({{\mathcal{P}}}, r
\right)$-\textsl{Rogers type} if there exists a constant $D = D_{\mathcal{P}, r
} \in \R_{>0}$ such that for any bounded Borel $E \subset \R^n$ with $m(E) > 0$ we have \eq{rogers}{ \norm{\widehat{\mathbbm{1}_E}^{{}_{{\mathcal{P}}}} - \left(\int_{X}\widehat{\mathbbm{1}_E}^{{}_{{\mathcal{P}}}} \,d\mu_X \right) \mathbbm{1}_X}_{r} \leq D \cdot 
m(E)
^{1/r}
.} 
\end{defn}

\begin{rmk} \label{siegelremark} \rm
\begin{itemize}
\item[ \rm (i)] The definition of ${{\mathcal{P}}}$-Siegel type is nothing more than the assertion that a variant of the Siegel Mean Value Theorem\textemdash first proved by 
Siegel in the context of $\ds \SL/\mathrm{SL}_n(\Z)$ in the seminal paper \cite{Siegel}\textemdash holds for the ${{\mathcal{P}}}$-Siegel transform on $X$. Using Lebesgue's Monotone Convergence Theorem, it is easy to see that if $G$ is of $\mathcal{P}$-Siegel type, then \equ{siegel} holds for any $L^1$ function $f: \R^n \to \R_{\geq 0}.$ 
Similarly, if there exists $r \in [1, \infty)$ for which $G$ is of $(\mathcal{P}, r)$-Rogers type, then \equ{rogers} is satisfied for any (not necessarily bounded) Borel $E\subset \R^n$ of finite measure.
\smallskip

\item[ \rm (ii)]
Assuming that $G$ is of ${{\mathcal{P}}}$-Siegel type, the assumption of $({{\mathcal{P}}}, 2)$-Rogers type is equivalent to the assumption that for any bounded Borel $E \subset \R^n,$ the variance of the random variable $\widehat{\mathbbm{1}_E}^{{}_{\mathcal{P}}}$ is bounded from above by a uniform scalar multiple of the expectation of $\widehat{\mathbbm{1}_E}^{{}_{\mathcal{P}}}.$ This condition was used by 
Schmidt to great effect in \cite{Schmidt}; see a remark after Theorem \ref{genericcounting} below. The definition of $({{\mathcal{P}}}, r)$-Rogers type for arbitrary $r \in [1, \infty)$ is a natural generalization of this condition.
\smallskip

\item[ \rm (iii)] Notice that if $G$ is of ${{\mathcal{P}}}$-Siegel type, then $G$ is of $({{\mathcal{P}}}, 1)$-Rogers type.            
\end{itemize}
\end{rmk}

Before we provide some examples of groups that satisfy the various Siegel and Rogers type axioms, let us record and prove some simple facts that will be helpful going forward. 

\begin{prop}[Logarithmic Convexity of $L^p$ Norms]\label{logconvex}
Let {$(Y,  \mu)$} be a measure space. Let $r, \, t \in \R_{\geq 1}$ and $\theta \in (0, 1)$ be arbitrary. Set $\ds s := \left( \frac{\theta}{r} + \frac{1-\theta}{t}\right)^{-1} \geq 1.$ For each $\ds  f \in L^{r}{(Y,\mu)} \cap L^t{(Y,\mu)}$ we then have \[ \|f\|_{s} \leq \|f\|_{r}^\theta \cdot \|f\|_t^{1-\theta}. \]
\end{prop}
\begin{proof}
This is a well-known special case of the Riesz\textendash Thorin interpolation theorem. For proofs of this special case and the general theorem, see \cite[Proposition 7.37]{EinWard} and \cite[Theorem 7.38]{EinWard}, respectively. \ignore{Let $\ds f \in L^p(Y) \cap L^t(Y).$ Note that we have $\ds r\left(\frac{\theta}{p} + \frac{1-\theta}{t} \right) = 1.$ Applying H\" older's inequality with the conjugate exponents $\ds \frac{p}{r\theta}$ and $\ds \frac{t}{r(1-\theta)},$ it follows  
\begin{align*}
    \norm{f}_r^r = \int_{Y} |f|^{r\theta} |f|^{r(1-\theta)} \,d\eta &\leq \left(\int_{Y} \left(|f|^{r\theta}\right)^{\frac{p}{r\theta}} \,d\eta\right)^{\frac{r\theta}{p}} \cdot \left(\int_{Y} \left(|f|^{r(1-\theta)}\right)^{\frac{t}{r(1-\theta)}} \,d\eta\right)^{\frac{r(1-\theta)}{t}} \\
    &= \left( \int_{Y} |f|^p \,d\eta \right)^{\frac{r\theta}{p}} \cdot \left( \int_{Y} |f|^t \,d\eta \right)^{\frac{r(1-\theta)}{t}}.
\end{align*} Hence,
\[ \|f\|_r \leq \left( \int_{Y} |f|^p \,d\eta \right)^{\frac{\theta}{p}} \cdot \left( \int_{Y} |f|^t \,d\eta \right)^{\frac{1-\theta}{t}} = \|f\|_p^\theta  \cdot \|f\|_t^{1-\theta}. \]} 
\end{proof}

\begin{cor}
Suppose that the group $G$ is of $({{\mathcal{P}}}, 1)$-Rogers type and that there exists $s \in  \R_{> 1}$ for which $G$ is of $({{\mathcal{P}}}, s)$-Rogers type. Then for each $r \in (1, s)$ the group $G$ is of $({{\mathcal{P}}}, r)$-Rogers type.
\end{cor}
\begin{proof}
Let $r \in (1, s).$ Fix $\theta \in (0, 1)$ for which $\ds 
\frac1r= \frac{\theta}{1} + \frac{1-\theta}{s}.$ Let $D_1 = D_{ {{\mathcal{P}}}, 1}$ and $D_{s} = D_{ {{\mathcal{P}}}, s}$ be as in Definition \ref{SiegelRogers}. Let $E \subset \R^n$ be a bounded Borel set. The foregoing proposition implies
\begin{align*}
\norm{\widehat{\mathbbm{1}_E}^{{}_{{\mathcal{P}}}} - \left(\int_{X}\widehat{\mathbbm{1}_E}^{{}_{{\mathcal{P}}}} \,d\mu_X \right) \mathbbm{1}_X}_r &\leq \norm{\widehat{\mathbbm{1}_E}^{{}_{{\mathcal{P}}}} - \left(\int_{X}\widehat{\mathbbm{1}_E}^{{}_{{\mathcal{P}}}} \,d\mu_X \right) \mathbbm{1}_X}_1^\theta \cdot \norm{\widehat{\mathbbm{1}_E}^{{}_{{\mathcal{P}}}} - \left(\int_{X}\widehat{\mathbbm{1}_E}^{{}_{{\mathcal{P}}}} \,d\mu_X \right) \mathbbm{1}_X}_{s}^{1-\theta} \\
&\leq D_1^\theta m(E)^{
{\theta}
} \cdot D_{s}^{1-\theta} m(E)^{\frac{1-\theta}{s}} 
= D_1^\theta D_{s}^{1-\theta} m(E)^{{1}/{r}}.
\end{align*}  \end{proof}
In this paper, the examples of $G$ that we shall consider are $\ASL,$ $\SL,$ and also $\Spn$ when $n$ is even. When $G = \ASL,$ it is clear that the only $\Gamma$-invariant subset of $\Z^n$ is $\Z^n$ itself. If $\ds G = \SL$ or $G = \Spn$ (for even $n$ only in the latter case), then $\Gamma$ acts transitively on $\Z_{\textrm{pr}}^n$; in these cases, two obvious choices of ${{\mathcal{P}}}$ are therefore $\ds {{\mathcal{P}}} = \Z_{\textrm{pr}}^n$ and $\ds {{\mathcal{P}}} = {\Z^n_{\ne0}} 
$. We now record the various Siegel and Rogers axioms that the groups just mentioned satisfy.

\smallskip

In the following theorem and thereafter, $\ds \zeta$ denotes the Euler\textendash Riemann zeta function. Let us mention that the following theorem is a compilation of results that are by now standard in the literature.


\begin{thm}\label{SiegelEg}
\begin{itemize}
    \item[ \rm (i)] The group $\ASL$ is of $\Z^n$-Siegel type with $\ds {c}_{ \Z^n} = 1.$
    \item[ \rm (ii)] The group $\SL$ is of $\ds \Z_{\mathrm{pr}}^n$-Siegel type with $\ds {c}_{ \Z_{\mathrm{pr}}^n} = 1/\zeta(n)$ and of ${\Z^n_{\ne0}}$-Siegel type with 
    $\ds {c}_{{\Z^n_{\ne0}}} = 1.$  
    \item[ \rm (iii)] Suppose $n$ is even. Then the group $\mathrm{Sp}_n(\R)$ is of $\ds \Z_{\mathrm{pr}}^n$-Siegel type with $\ds {c}_{ \Z_{\mathrm{pr}}^n} = 1/\zeta(n)$ and of $\Z^n_{\ne0}$-Siegel type with $\ds {c}_{\Z^n_{\ne0}} = 1.$
\end{itemize}
\end{thm}
\begin{proof}
\begin{itemize} 
    \item[ \rm (i)] From Lemma 3 of \cite{Affine} and the ensuing discussion therein, we see that this claim holds with $\ds {c}_{ \Z^n} = 1.$
   
\smallskip
    \item[ \rm (ii)] {By the main theorem in \cite{Siegel} and \cite[(25)]{Siegel}, it follows that for every bounded and compactly supported Riemann integrable function $f : \R^n \to  \R_{\geq 0},$ we have \[ \int_{\R^n}f \, dm = \int_{X}\widehat{f}^{{\,}_{\Z^n_{\ne0}}} \, d\mu_{X}\] and
    \[ \int_{\R^n} f \, dm = \zeta(n)\int_{X}\widehat{f}^{{\,}_{\Z_{\mathrm{pr}}^n}} \,d\mu_{X}. \] The desired results now follow from Lebesgue's Monotone Convergence Theorem.} 
  
\smallskip  
    \item[ \rm (iii)] After making the requisite changes in notation, the assertion that $\mathrm{Sp}_n(\R)$ is of $\ds \Z_{\mathrm{pr}}^n$-Siegel type with  $\ds {c}_{ \Z_{\mathrm{pr}}^n} = 1/\zeta(n)$ is precisely the content of \cite[(0.6)]{KY1}. Let $f: \R^n \to \R_{\geq 0}$ be a compactly supported Borel measurable function. For any $k \in \Z,$ define $f_k : \R^n \to \R_{\geq 0}$ by $f_k(\xx) := f(k\xx).$ Then for any $\Lambda \in X,$
    \[ \widehat{f}^{{\,}_{\Z^n_{\ne0}}}(\Lambda)= \sum_{\vv \in \Lambda \ssm \{\mathbf{0}\}} f(\vv) = \sum_{k=1}^\infty \sum_{\vv \in \Lambda_{\mathrm{pr}}} f(k\vv) = \sum_{k=1}^\infty \sum_{\vv \in \Lambda_{\mathrm{pr}}} f_k(\vv) = \sum_{k=1}^\infty \widehat{f_k}^{{}_{\Z_{\mathrm{pr}}^n}}(\Lambda). \] It is now easy to conclude that $\mathrm{Sp}_n(\R)$ is of $\ds \Z_{\ne0}^n$-Siegel type with   $\ds {c}_{ \Z_{\neq 0}^n} = 1$. \end{itemize} \end{proof}
    
\begin{thm}\label{RogersEg}
\begin{itemize}
\item[ \rm (i)] The group $\ASL$ is of $\left(\Z^n, 2\right)$-Rogers type.
\item[ \rm (ii)] Suppose $n \geq 3.$ Then $\SL$ is of $\ds \left(\Z_{\mathrm{pr}}^n, 2\right)$-Rogers type and of $(\Z^n_{\ne0},2)$-Rogers type.
\item[ \rm (iii)] Suppose $n$ is even and $n \geq 4.$ Then $\mathrm{Sp}_n(\R)$ is of $\ds \left(\Z_{\mathrm{pr}}^n, 2\right)$-Rogers type and of $(\Z^n_{\ne0},2)$-Rogers type. \end{itemize}  \end{thm}
\begin{proof}
\begin{itemize}
    \item[ \rm (i)] The result \cite[Lemma 4]{Affine} shows that for any bounded Borel $E \subset \R^n,$ we have \[ \norm{\widehat{\mathbbm{1}_E}^{{}_{\Z^n}} - \left(\int_{X}\widehat{\mathbbm{1}_E}^{{}_{\Z^n}} \,d\mu_X \right) \mathbbm{1}_X}_2 = m(E)^{{1}/{2}}. \] 
    
    \item[ \rm (ii)] Let $E \subset \R^n$ be bounded and Borel. Since $\ds \int_{X}\widehat{\mathbbm{1}_E}^{{}_{\Z^n_{\mathrm{pr}}}} \,d\mu_{X} = \frac{1}{\zeta(n)}m(E)$,
    a simple change of notation and a routine algebraic manipulation of
    \cite[(0.2)]{KY1} yield \[ \norm{\widehat{\mathbbm{1}_E}^{{}_{\Z^n_{\mathrm{pr}}}} - \frac{m(E)}{\zeta(n)}\mathbbm{1}_X}_2 \leq \sqrt{\frac{2}{\zeta(n)}} \,m(E)^
    {{1}/{2}}. \] Hence, $\SL$ is of $(\Z_\mathrm{pr}^n, 2)$-Rogers type. 
    
    \smallskip
    
    Let $B$ denote the closed Euclidean ball in $\R^n$ that is centered at the origin and whose measure is equal to $m(E).$ By \cite[Theorem 1 and Lemma 1]{RogSet}, it follows \[ \norm{\widehat{\mathbbm{1}_E}^{{}_{\Z^n_{\ne0}}}}_2^2 \leq \norm{\widehat{\mathbbm{1}_B}^{{}_{\Z^n_{\ne0}}}}_2^2 \leq m(E)^2 + \sum_{k, q \in \Z_{\neq 0} : \gcd(k, q) = 1} \int_{\R^n} \mathbbm{1}_B(k\xx) \,  \mathbbm{1}_B(q\xx) \, dm(\xx). \] As in the proof of \cite[Theorem 2.2]{Log}, we have \[ \sum_{k, q \in \Z_{\neq 0} : \gcd(k, q) = 1} \int_{\R^n} \mathbbm{1}_B(k\xx) \,  \mathbbm{1}_B(q\xx) \, dm(\xx) \leq 8 \frac{\zeta(n-1)}{\zeta(n)} m(E).  \] Hence, $\SL$ is of $(\Z^n_{\ne0}, 2)$-Rogers type. 
      \item[ \rm (iii)]
      Let $E \subset \R^n$ be bounded and Borel. Since $\ds \int_{X}\widehat{\mathbbm{1}_E}^{{}_{\Z_{\textrm{pr}}^n}} \,d\mu_X = \frac{1}{\zeta(n)}m(E),$ a simple change of notation and a routine rearrangement of \cite[(0.10)]{KY1} yield \[ \norm{\widehat{\mathbbm{1}_E}^{{}_{\Z_{\textrm{pr}}^n}} - \left(\int_{X}\widehat{\mathbbm{1}_E}^{{}_{\Z_{\textrm{pr}}^n}} \,d\mu_X \right) \mathbbm{1}_X}_2 \leq \frac{2}{\sqrt{\zeta(n)}}\,m(E)^{{1}/{2}}. \]  
    Since $\ds \int_{X}\widehat{\mathbbm{1}_E}^{{}_{\Z^n_{\ne0}}} \,d\mu_X = m(E),$ a simple change of notation and a routine rearrangement of \cite[(0.11)]{KY1} yield \[ \norm{\widehat{\mathbbm{1}_E}^{{}_{{}_{\Z^n_{\ne0}}}} - \left(\int_{X}\widehat{\mathbbm{1}_E}^{{}_{\Z^n_{\ne0}}} \,d\mu_X \right) \mathbbm{1}_X}_2 \leq \frac{2 \zeta\left(\frac{n}{2}\right)}{\sqrt{\zeta(n)}}\,m(E)^{{1}/{2}}. \] \end{itemize} \end{proof} Before handling the case of $\mathrm{SL}_2(\R),$ we first prove an interpolation result that we shall have to use.
\begin{lem}\label{interpolation}
Let $G$ be a closed subgroup of $\SL,$ and let $\Gamma$ be as in \equ{gamma}. Suppose further that $G$ is of $\ds \Z_{\mathrm{pr}}^n$-Siegel type with $\ds {c}_{ \Z_{\mathrm{pr}}^n} = 1/\zeta(n)$ and of $\ds (\Z_{\mathrm{pr}}^n, 2)$-Rogers type. For each $r \in (1, 2)$ it then follows that $G$ is of $\ds \left(\Z^n_{\ne 0}, r\right)$-Rogers type.
\end{lem}
\begin{proof} 
Arguing as in (iii) of Theorem \ref{SiegelEg}, we conclude that $G$ is of $\Z^n_{\ne0}$-Siegel type with $\ds {c}_{\Z^n_{\ne0}} = 1.$ Let $\ds D = D_{ \Z_{\mathrm{pr}}^n, 2} \in \R_{>0}$ be as in Definition \ref{SiegelRogers}. Let $A \subset \R^n$ be bounded and Borel.
Then \[ \norm{\widehat{\mathbbm{1}_A}^{{}_{\Z_{\textrm{pr}}^n}} - \frac{m(A)}{\zeta(n)} \mathbbm{1}_X}_2 \leq D \, m(A) ^{{1}/{2}}. \] Since $G$ is of $\ds \Z_{\mathrm{pr}}^n$-Siegel type with $\ds {c}_{ \Z_{\mathrm{pr}}^n} = 1/\zeta(n),$ we have \[ \norm{\widehat{\mathbbm{1}_A}^{{}_{\Z_{\textrm{pr}}^n}} - \frac{m(A)}{\zeta(n)} \mathbbm{1}_X}_1 \leq \frac{2m(A)}{\zeta(n)}. \] Let $r \in (1, 2)$ be given. Set $\ds \theta := \frac{2}{r} - 1$; then $\theta \in (0, 1)$ and $\ds r = \left( \frac{\theta}{1} + \frac{1-\theta}{2} \right)^{-1}.$ By the logarithmic convexity of the $L^p$ norms, one has
\begin{align*}
    \norm{\widehat{\mathbbm{1}_A}^{{}_{\Z_{\textrm{pr}}^n}} - \frac{m(A)}{\zeta(n)} \mathbbm{1}_X}_{r} &\leq \norm{\widehat{\mathbbm{1}_A}^{{}_{\Z_{\textrm{pr}}^n}} - \frac{m(A)}{\zeta(n)} \mathbbm{1}_X}_1^\theta \cdot \norm{\widehat{\mathbbm{1}_A}^{{}_{\Z_{\textrm{pr}}^n}} - \frac{m(A)}{\zeta(n)} \mathbbm{1}_X}_2^{1-\theta} \\
    &\leq \left(\frac{2}{\zeta(n)}\right)^\theta m(A)^\theta \cdot D^{1-\theta} m(A)^\frac{1-\theta}{2} 
    = \frac{2^\theta D^{1-\theta}}{\zeta(n)^\theta} \,m(A)^{{1}/{r}}.  
\end{align*}
Now let $E \subset \R^n$ be bounded and Borel. For each $k \in \Z_{>0},$ let $$\ds E_k := \{ \xx \in \R^n : k \xx \in E \}.$$ We then have
\begin{align*}
    \norm{\widehat{\mathbbm{1}_E}^{{}_{\Z^n_{\ne0}}} - m(E) \mathbbm{1}_X }_r &= \norm{ \sum_{k=1}^\infty \left(\widehat{\mathbbm{1}_{E_k}}^{{}_{\Z_{\textrm{pr}}^n}} - \frac{m(E_k)}{\zeta(n)} \mathbbm{1}_X \right)}_{r} \leq \sum_{k=1}^\infty \norm{ \widehat{\mathbbm{1}_{E_k}}^{{}_{\Z_{\textrm{pr}}^n}} - \frac{m(E_k)}{\zeta(n)} \mathbbm{1}_X }_{r} \\
    &\leq \sum_{k=1}^\infty \frac{2^\theta D^{1-\theta}}{\zeta(n)^\theta} m(E_k)^{{1}/{r}}
    = \left( \frac{2^\theta D^{1-\theta}}{\zeta(n)^\theta} \sum_{k=1}^\infty k^{- {n}/{r}} \right) m(E)^{{1}/{r}}.
\end{align*}
Since $\ds 1 < r < 2 \leq n,$ it follows that $\ds D_{\Z^n_{\ne0}, r} :=   \left( \frac{2^\theta \left(D_{ \Z_{\mathrm{pr}}^n, 2}\right)^{1-\theta}}{\zeta(n)^\theta} \sum_{k=1}^\infty k^{- {n}/{r}}\right) < \infty.$  \end{proof}  

\begin{thm}\label{SL2}
The group $\mathrm{SL}_2(\R)$ is of $\ds \left(\Z_{\mathrm{pr}}^n, 2\right)$-Rogers type; for each $\ds r \in (1, 2)$ the group $\mathrm{SL}_2(\R)$ is of $\left(\Z^n_{\ne0}, r \right)$-Rogers type.
\end{thm}  
\begin{proof}
If $E$ is any bounded Borel subset of $\R^n$ that has sufficiently large volume, then \cite[(4.4)]{Log} yields \[ \norm{\widehat{\mathbbm{1}_E}^{{}_{\Z_{\textrm{pr}}^n}} - \frac{m(E)}{\zeta(2)}\mathbbm{1}_X}_2 \leq 4 m(E)^{{1}/{2}}. \] This implies the first assertion by choosing the constant $\ds D_{\Z_{\textrm{pr}}^n, 2}$ of Definition \ref{SiegelRogers} (ii) to be sufficiently large. The second assertion now follows at once from Lemma \ref{interpolation}. \end{proof}  


Now that we have considered some examples of groups that satisfy the Siegel and Rogers axioms, let us state and prove the first results that make these axioms worthwhile.               


\begin{thm}\label{genericcounting}
Let $G$ be a closed subgroup of $\ASL$, let $\Gamma$ be as in \equ{gamma}, and let ${{\mathcal{P}}}$ be a $\Gamma$-invariant subset of $\Z^n.$ Suppose $G$ is of $\mathcal{P}$-Siegel type with $c=c_{\mathcal{P}}$. Let $E$ be a Borel measurable subset of $\R^n.$ 
\smallskip  
\begin{itemize}
\item[\rm (i)] If $m(E) < \infty,$ then $\ds \mu_X\left(\left\lbrace \Lambda \in X: \mathrm{card}\left(\Lambda_{{\mathcal{P}}} \cap E\right) < \infty \right\rbrace\right) = 1.$   
\end{itemize}
\smallskip
For the remaining statements of this theorem, suppose in addition to the preceding hypotheses that we are given $r \in \R_{> 1}$ for which $G$ is of $\ds \left(\mathcal{P}, r \right)$-Rogers type. 
\smallskip
\begin{itemize}
\item[\rm (ii)] Suppose $m(E) = \infty.$ Let $\norm{\cdot}$ be a norm on $\R^n$ and for each $t \in \R_{>0},$ set $$\ds E_t := \{\xx \in E : \norm{\xx} \leq t\}.$$ Then 
for $\mu_X$-almost every $\Lambda \in X,$ one has $\ds \lim_{t \to \infty } \frac{\card\left(\Lambda_{\mathcal{P}} \cap E_{t}\right)}{c\, m\left(E_{t}\right)} = 1.$ In particular, $\ds \mu_X\left(\left\lbrace \Lambda \in X: \mathrm{card}\left(\Lambda_{{\mathcal{P}}} \cap E\right) = \infty \right\rbrace\right) = 1.$   
\smallskip  
\item[\rm (iii)] Let $\{F_k\}_{k \in \Z_{\geq 1}}$ be 
Borel measurable subsets of $\R^n$ with $0 < m(F_k) < \infty$ for each $k \in \Z_{\geq 1}$. Suppose $\ds \sum_{k=1}^{\infty} m(F_k)^{1-r} < \infty.$ Then the following holds: for $\mu_X$-almost every $\Lambda \in X,$ there exists some $k_\Lambda \in \Z_{\geq 1}$ such that for each integer $k \geq k_\Lambda,$ we have $\ds \Lambda_{\mathcal{P}} \cap F_k \neq \varnothing.$ 
\end{itemize}           
\end{thm}
\begin{proof}
\begin{itemize}
\item[\rm (i)] Suppose $\ds m(E) < \infty.$ Apply \equ{siegel} to $\ds \mathbbm{1}_E : \R^n \to \R_{\geq 0}$; this is valid in light of Remark \ref{siegelremark}(i). This shows that for $\mu_X$-almost every $\Lambda \in X,$ one has $\ds \widehat{\mathbbm{1}_E}^{{}_{{\mathcal{P}}}}(\Lambda) = \mathrm{card}\left(\Lambda_{{\mathcal{P}}} \cap E\right) < \infty.$ 

\smallskip

\item[\rm (ii)] For each $\ds t \in \R_{>0},$ define $\ds h_t : X \to \R_{\geq 0}$ by $$\ds h_t(\Lambda) := \widehat{\mathbbm{1}_{E_t}}^{{}_{{\mathcal{P}}}}(\Lambda) = \mathrm{card}\left(\Lambda_{{\mathcal{P}}} \cap E_t\right).$$ For each $t \in \R_{>0},$ one has $\ds \int_X h_t \ d\mu_X = c\, m(E_t).$ Let $D = D_{\mathcal{P}, r}$ be as in Definition \ref{SiegelRogers}. Fix any $\gamma \in \R$ with $\ds \gamma > (r-1)^{-1}$. For each $k \in \Z_{\geq 1},$ fix $t_k \in \R_{>0}$ for which $m\left(E_{t_k}\right) = k^\gamma.$ Let $\varepsilon \in \R_{>0}$ be given. For each $k \in \Z_{\geq 1},$ it follows from Markov's inequality and the hypotheses on $G$ that we have 
\begin{align*}
\mu_X\left(\left\lbrace \Lambda \in X : \left| \frac{h_{t_k}(\Lambda)}{ c\, m(E_{t_k})} - 1 \right| \geq \varepsilon \right\rbrace\right) &\leq \frac{1}{\varepsilon^r} \left\|\frac{h_{t_k}}{ c\, m\left(E_{t_k}\right)} - \mathbbm{1}_{X}\right\|_{r}^{r} \\ &\leq \frac{1}{\varepsilon^r} \, \left(\frac{D}{c}\right)^r m\left(E_{t_k}\right)^{1-r} = \left(\frac{D}{\varepsilon \, c }\right)^r k^{\gamma(1-r)}.
\end{align*}

Since $\ds \varepsilon \in \R_{>0}$ is arbitrary and $\ds \gamma(1-r) < -1,$ the Borel\textendash Cantelli lemma now implies that for $\mu_X$-almost every $\Lambda \in X,$ we have $\ds  \lim_{k \to \infty } \frac{\card\left(\Lambda_{\mathcal{P}} \cap E_{t_k}\right)}{c\, m\left(E_{t_k}\right)} = 1.$ For any $k \in \Z_{\geq 1}$ and any $\ds t \in \left[ t_k, t_{k+1} \right),$ we have
\begin{align*}
\frac{k^\gamma}{(k+1)^\gamma}\frac{h_{t_k}}{ m(E_{t_k})} = \frac{h_{t_k}}{ m(E_{t_{k+1}})}\le \frac{h_{t}}{ m(E_{t})} \le  \frac{h_{t_{k+1}}}{ m(E_{t_k})} =   \frac{(k+1)^\gamma}{k^\gamma}\frac{h_{t_{k+1}}}{ m(E_{t_{k+1}})}.
\end{align*} The result follows. 

\smallskip

\item[(iii)] Let $\ds D = D_{\mathcal{P}, r}$ be as in Definition \ref{SiegelRogers}. 
For each $\ds k \in \Z_{\geq 1}$ we have 
\begin{align*}\mu_X\left( \left\lbrace \Lambda \in X : \Lambda_{{\mathcal{P}}} \cap F_k = \varnothing \right\rbrace \right) &\leq 
\mu_X\left( \left\lbrace \Lambda \in X : \left|\widehat{\mathbbm{1}_{F_k}}^{{}_{{\mathcal{P}}}} (\Lambda) - c\, m(F_k)\right|^{r} \ge \left(c\, m(F_k)\right)^{r} \right\rbrace \right) \\ 
&\leq \norm{\widehat{\mathbbm{1}_{F_k}}^{{}_{{\mathcal{P}}}} - c\, m(F_k) \mathbbm{1}_X}_{r}^{r} \big( c\, m(F_k) \big)^{-r} \\
&\leq D^{r} m(F_k) \big( c\, m(F_k) \big)^{-r} 
= D^{r} c^{-r} m(F_k)^{1-r}.
\end{align*} The desired result now follows from the Borel\textendash Cantelli lemma. \end{itemize} \end{proof} 
Statement (ii) of the foregoing theorem is a variation of a very general counting result due to 
Schmidt: see \cite{Schmidt}. See also \cite[Chapter 1, Lemma 10]{Sprind} for a result abstracted by 
Sprind\v zuk  from the work of 
Schmidt. Following Sprind\v zuk, it is not difficult to state and prove {part} (ii) of the above theorem with an estimate for an error term. Let us also mention that our proof of (ii) is similar to an argument used by 
Durrett in his proof of \cite[Chapter 1, Theorem 6.8]{Durrett}.  

\begin{rmk}\label{Minkowski} \rm
Let $G$ be a closed subgroup of $\ASL$, let $\Gamma$ be as in \equ{gamma}, and let ${{\mathcal{P}}}$ be a $\Gamma$-invariant subset of $\Z^n.$ Suppose that $G$ is of $\mathcal{P}$-Siegel type, and suppose that there exists $r \in \R_{> 1}$ for which $G$ is of $\ds \left(\mathcal{P}, r\right)$-Rogers type. It is now {easy} to prove a probabilistic analogue of the Minkowski convex body theorem. Indeed, let $\ds c = c_\mathcal{P}$ and $\ds D = D_{\mathcal{P}, r}$ be as in Definition \ref{SiegelRogers}; let $E$ be a Borel subset of $\R^n$ with $\ds 0 < m(E) < \infty.$ As in the proof of Theorem \ref{genericcounting}(iii), it follows that
\[ \mu_X\left( \left\lbrace \Lambda \in X : \Lambda_{{\mathcal{P}}} \cap E = \varnothing \right\rbrace \right) \leq D^r c^{-r} m(E)^{1-r}. \] This sort of result, {with $r=2$,} was first established by Athreya\textendash Margulis 
for $G = \mathrm{SL}_n(\R)$ and $\mathcal{P} = \Z^n_{\ne0}$ 
 \cite[Theorem 2.2]{Log}, {and then}
by Athreya 
for $G = \ASL$ and $\mathcal{P} = \Z^n$ 
\cite[Theorem 1]{Affine}.
\end{rmk}

\begin{rmk}\rm 
{Suppose $n \geq 2$ is arbitrary and $G = \SL.$ In \cite[Corollary 2.14]{Str}, Str\"ombergsson proves that the bound of Athreya\textendash Margulis in \cite[Theorem 2.2]{Log} is sharp. It now follows from the preceding remark that for each $r \in \R_{>2}$ and each subset $\mathcal{P}$ of $\Z^n_{\ne0}$ that is $\Gamma = \mathrm{SL}_n(\Z)$-invariant, we have that the group $G$ is \textsl{not} of $\left(\mathcal{P}, r\right)$-Rogers type. } 
\end{rmk}

{\begin{rmk}\rm 
Suppose $n \geq 2$ is arbitrary and $G = \R^n,$ which is a closed subgroup of $\ASL.$ Then $\Gamma = \Z^n$, and $X = \R^n/\Z^n$ 
is the $n$-dimensional torus.
It is easy to see that $\R^n$ is then of $\Z^n$-Siegel type with $c_{\Z^n} = 1.$ For each $r \in \R_{>1},$ however, $\R^n$ is is \textsl{not} of $\left(\Z^n, r\right)$-Rogers type. This may be seen from the spectacular impossibility of obtaining a result as in Remark \ref{Minkowski}. For any $\varepsilon \in (0, 1),$ define $\ds U_\varepsilon := \R^{n-1} \times \left( \frac{1}{2} - \frac{\varepsilon}{2}, \frac{1}{2} + \frac{\varepsilon}{2}\right) \subseteq \R^{n-1} \times (0, 1).$ For each $\varepsilon \in (0, 1),$ we then have $\ds m(U_\varepsilon) = \infty$ and $\ds \mu_X\left( \left\lbrace \Lambda \in X : \Lambda_{\Z^n} \cap U_\varepsilon = \varnothing \right\rbrace \right) = 1 - \varepsilon.$
\end{rmk}  }

\smallskip 

In the following section, we transfer our counting results for generic lattices to statements involving small values of generic functions, thereby establishing a more general version of Theorem \ref{simplemainapproxthm}.      

\section{Zero\textendash full laws in Diophantine approximation} \label{zerofull}
We begin by proving two lemmata. 

\begin{lem}\label{normimmmeas}
Let $\ds f 
: \R^n \to \R^{\ell}$ be subhomogeneous,  and let $\ds \psi 
: \R_{\geq 0} \to \left(\R_{>0}\right)^\ell$ be regular and nonincreasing.  {Let $\eta$ and $\nu$ be any norms on $\R^n$, and let $s \in \R_{>0}.$} Then
{$m(A_{f, s\psi, \eta}) < \infty$ 
if and only if $m(A_{f, \psi, \nu}) < \infty$}.
\end{lem}        
\begin{proof} 
Suppose without loss of generality that the image of $f$ is a subset of $\left(\R_{\geq 0}\right)^\ell.$
Let $\ds a = a_\psi$, $\ds b = b_\psi,$ and $\ds d = d_f$ be as in Definition \ref{basicapdefns}; let $\xx \in A_{f, \psi, \eta}$. Let $s \in \R_{>0}$ be given. Suppose first $s \leq 1.$ Then $\ds s^{1/d} \in (0, 1],$ and thus \[  f\big(s^{1/d} \xx \big) \les s f(\xx) \les s\psi\big(\eta(\xx)\big) \les s\psi\big(\eta(s^{1/d}\xx)\big). \]  This proves $\ds s^{1/d} A_{f, \psi, \eta} \subseteq A_{f, s\psi, \eta}.$ Also, note that $\ds A_{f, s\psi, \eta} \subseteq A_{f, \psi, \eta}$. Hence, the Lebesgue measure of $A_{f, s\psi, \eta}$ is finite if and only if the Lebesgue measure of $A_{f, \psi, \eta}$ is finite. Suppose next $s \geq 1.$ By repeating the preceding argument with $s \psi$ in place of $\psi$ and $s^{-1}$ in place of $s,$ we obtain the same conclusion for $s \geq 1.$

Now, using the equivalence of the two norms, fix $C \in \R_{> 1}$ for which $\ds C^{-1} \eta \leq \nu \leq C\eta.$ Fix a positive integer ${k}$ for which $\ds a^{k} > C.$ Suppose that $A_{f, \psi, \nu}$ has infinite Lebesgue measure. Let $\ds \xx \in A_{f, \psi, \nu}.$ By a simple induction, \[ f(\xx) \les \psi\big(\nu(\xx)\big) \les \psi\left(C^{-1}\eta(\xx)\right) \les b^{-{k}} \psi\big(a^{k} C^{-1}\eta(\xx)\big) \les b^{-{k}} \psi\big(\eta(\xx)\big). \] Thus, the Lebesgue measure of $\ds A_{f, b^{-{k}}\psi, \eta}$ is infinite as well. In conjunction with the foregoing and by symmetry, this completes the proof.  \end{proof}

\begin{lem}\label{sublinearnew}
Let $\ds \psi 
: \R_{\geq 0} \to \left(\R_{>0}\right)^{\ell}$ be regular and nonincreasing. Then the following holds: for any $c \in \R_{\geq 0}$ there exists $s \in \R_{>0}$ such that for each $x \in [0, c]$ and each $y \in \R_{>c},$ one has $\ds \psi(y-x) \les s\psi(y).$
\end{lem}
\begin{proof}
Let $\ds a = a_\psi$ and $\ds b = b_\psi$ be as in Definition \ref{basicapdefns}.
Let $c \in \R_{\geq 0}.$ Define \[ s := \max_{1 \leq i \leq {\ell}}\left(\frac{1}{b}, \frac{\psi_{i}(0)}{\psi_{i}( \frac{ac}{a-1})}\right). \]
Let $x \in [0, c]$ and $y \in \R_{> c}.$ We consider two cases.
\begin{itemize}
    \item Case 1: suppose $\ds y \leq \frac{ac}{a-1}.$ Then
    \[ \psi(y-x) \les \psi(0) \les s\psi\left( \frac{ac}{a-1}\right) \les s \psi(y). \] 
    \item Case 2: suppose $\ds y > \frac{ac}{a-1}.$ Since $\ds c \geq x \geq 0$ and $a - 1 >0,$ it follows
  $\ds y > \frac{ax}{a-1}$; hence,
        $ \ds y-x > \frac{y}{a}$.
    Thus, \[ \psi(y-x) \les \psi\left(\frac{y}{a}\right) \les \frac{1}{b} \psi\left( a \cdot \frac{y}{a} \right) = \frac{1}{b} \psi(y) \les s\psi(y). \] \end{itemize} This completes the proof. \end{proof}
    
Before proving our main results, let us first augment two definitions given in \S 1.

\begin{defn}\label{twobasicapdefnes}
In this definition, assume that we are using the same notation as in Definition \ref{basicapdefns}. Now {take  an arbitrary subset $\mathcal{P}$ of $\Z^n$ and}
   \begin{itemize}
   \item 
   say that $f$ is $(\psi, \nu, \mathcal{P})$-\textsl{approximable} if $A_{f, \psi, \nu} \cap \mathcal{P}$ has infinite cardinality; 
   \smallskip
   
   \item 
   say that $f$ is \textsl{uniformly} $(\psi, \nu, \mathcal{P})$-\textsl{approximable} if
    {$B_{f, \psi(T), \nu, T} \cap \mathcal{P}  \neq \varnothing$ for each sufficiently large $T \in \R_{> 0}$}.
   \end{itemize}
\end{defn}
    
{Notice that by taking $\mathcal{P} = \Z^n_{\ne0}$ in the above definition  we recover the previously defined notions of asymptotic and uniform $(\psi, \nu)$-approximability.
    We now state and prove our main result on asymptotic approximation. 

\begin{thm}\label{mainapproxthmnew}
Let $G$ be a closed subgroup of $\ASL$, let $\Gamma$ be as in \equ{gamma}, and let ${{\mathcal{P}}}$ be a $\Gamma$-invariant subset of $\Z^n.$ Suppose $G$ is of $\mathcal{P}$-Siegel type. Let $\ds {c} = {c}_{{{\mathcal{P}}}}$ be as in Definition \ref{SiegelRogers}(i). Let $\ds \psi 
: \R_{\geq 0} \to \left(\R_{>0}\right)^{\ell}$ be regular and nonincreasing; let $\ds f 
: \R^n \to \R^{\ell}$ be subhomogeneous; let $\nu$ and $\eta$ be norms on  $\R^n$. 

\smallskip

\begin{itemize}
\item[\rm (i)]  Suppose $\ds m\left(A_{f, \psi, \eta}\right) < \infty.$ Then for almost every  $g \in G$ the function $f\circ g$ is not 
$(\psi, \nu, {{\mathcal{P}}})$-approximable. 
\smallskip

\item[ \rm(ii)]     
Suppose $\ds m\left(A_{f, \psi, \eta
}\right) = \infty$, and suppose there exists $r \in \R_{> 1}$  for which $G$ is of $\ds \left(\mathcal{P}, r\right)$-Rogers type. Then for each nonempty 
 compact subset $K$ of $G$  there exist constants $D_K \in \R_{\geq 1},$ $E_K \in \R_{\geq 0}$  and $\ds J_K \in \R_{\geq 1}$ such that for $\mu_G$-almost every $g \in K$ we have 
 \begin{align}\label{limsup}
\limsup_{T \to \infty}\frac{\card\left\lbrace \vv \in {{\mathcal{P}}} : (f \circ  g )(\vv) \les \psi\big( \nu(\vv) \big) \ \mathrm{and} \ 2 D_K E_K < \nu(\vv) \leq T \right\rbrace}{m\left(\left\lbrace \ttt \in \R^n : f(\ttt) \les J_K \psi\big(\nu(\ttt)\big) \ \mathrm{and} \ {E_K} < \nu(\ttt) \leq D_K T + E_K \right\rbrace\right)} \leq {{c}_{{\mathcal{P}}}},\qquad \end{align} and 
 \begin{align}\label{liminf}\qquad\ \ 
\liminf_{T \to \infty} \frac{\card\left\lbrace \vv \in {{\mathcal{P}}} : (f \circ  g )(\vv) \les \psi\big( \nu(\vv)\big) \ \mathrm{and} \ {E_K D_K^{-1}} < \nu(\vv) \leq D_K E_K + D_K T \right\rbrace}{m\left(\left\lbrace \ttt \in \R^n : f(\ttt) \les J_K^{-1}\psi\big((\nu(\ttt)\big) \ \mathrm{and} \ 2E_K < \nu(\ttt) \leq T \right\rbrace\right)} \geq {{c}_{{\mathcal{P}}}}. \end{align} 
Moreover, if $K \subseteq \SL,$ then each of the above inequalities holds with $E_K = 0.$ \\
In particular, for almost every $g \in G$ the function $f\circ g$ is $(\psi, \nu, {{\mathcal{P}}})$-approximable. 

\end{itemize}          
\end{thm}    
\begin{proof} 
Let us denote elements of $\ASL$ by $\langle h,\zz\rangle$, where $h\in \SL$ and $\zz\in\R^n$; that is, \eq{affine}{\langle h,\zz\rangle : \R^n \to \R^n\text{ is the affine transformation given by }\xx \mapsto h\xx + \zz.} We suppose without loss of generality that the image of $f$ is a subset of $\left( \R_{\geq 0}\right)^\ell.$ 
For any $h \in \SL$ we write $\|h\|$ to denote the operator norm of $h$ when both the domain and codomain of $h$ are equipped with the norm $\nu$ on $\R^n$ that is mentioned in the hypotheses. 

\smallskip    
{Suppose that $\ds m\left(A_{f, \psi, \eta}\right) < \infty.$} Lemma \ref{normimmmeas} implies that for any $N \in \Z_{\geq 1}$ we have $\ds m\left(A_{f, N\psi, \nu}\right) < \infty.$ Theorem \ref{genericcounting}(i) then implies $$\mu_{X}\left(\left\{\Lambda \in X: \card\left(\Lambda_{{\mathcal{P}}} \cap A_{f, N\psi, \nu}\right) = \infty\right\}\right) = 0,$$ which is equivalent to 
\eq{implication}{ \mu_{G}\left(\left\{g \in G: \card\left(g\mathcal{P} \cap A_{f, N\psi, \nu}\right) = \infty\right\}\right) = 0.} Let $\ds a = a_\psi,$ $\ds b = b_\psi,$ and $\ds d = d_f$ be as in Definition \ref{basicapdefns}. Let $g = \langle h, \zz \rangle$ be any element of $G$ for which 
\eq{appr}{f \circ g\text{ is $\left(\psi, \nu, {{\mathcal{P}}}\right)$-approximable}.} Let $\ds D := \max\left\lbrace \|h\|, \norm{h^{-1}}\right\rbrace,$ and let $\ds E := \nu(\zz).$ Let ${k}$ be a nonnegative integer for which $\ds a^{k} \geq D.$ Let $\ds C := b^{-{k}}.$ Appealing to Lemma \ref{sublinearnew}, we let $F \in \R_{>0}$ be a constant for which the following is true: for each $x \in [0, E]$ and each $\ds y \in (E, \infty),$ we have $\psi(y-x) \les F \psi(y).$ Finally, let $N$ be any integer with $N \geq CF.$

Let $\vv$ be an arbitrary element of the infinite set $\ds \left\lbrace \xx \in   \mathcal{P} \cap A_{f \circ g, \psi, \nu}   : \nu(\xx) > 2 DE \right\rbrace.$
Notice that \[ \norm{h^{-1}} \nu(h\vv) \geq \nu(\vv) > 2D E \geq 2 \norm{h^{-1}} E, \] whence $\ds \nu(h\vv) > 2E.$ Hence, \[ \nu(h\vv+ \zz) \geq \nu(h\vv) - \nu(\zz) > 2E - \nu(\zz) = E. \]

Since $\ds (f \circ g)(\vv) \les \psi\big( \nu(\vv)\big),$ it follows  
\begin{align*}
         f(g\vv) \les \psi\big(\nu(\vv)\big) &\les \psi\left( \frac{\nu(h\vv)}{\|h\|}\right) \les b^{-{k}} \psi\left( a^{k} \frac{\nu(h\vv)}{\|h\|}\right) \les C \psi\big(\nu(h\vv)\big) 
         \\         &\les C \psi\big( \nu(h\vv + \zz) - \nu(\zz) \big) 
         \les C F \psi\big( \nu(h\vv + \zz) \big) \\
         &
         \les N\psi\big( \nu(h\vv + \zz) \big) 
         = N\psi\big( \nu(g\vv) \big).
    \end{align*} Thus, $\card\left(g{{\mathcal{P}}} \cap A_{f, N\psi, \nu}\right) = \infty;$ hence, in view of \equ{implication}, the set of $g\in G$ that satisfy \equ{appr} is null. This proves (i). 
    
\smallskip

{Suppose now  that $\ds m\left(A_{f, \psi, \eta}\right) = \infty$; suppose further that there exists $r \in \R_{> 1}$  for which $G$ is of $\ds \left(\mathcal{P}, r\right)$-Rogers type.} Let $\ds \varepsilon \in \R_{>0}$ be given. Let $K$ be an arbitrary nonempty compact subset of $G$. Since the inversion map is a homeomorphism and finite unions of compact sets are compact, we assume without loss of generality that $K = K^{-1}$. Define $\pi : \ASL \to \SL$ and $\rho : \ASL \to \R^n$ by $\ds \pi : \langle h, \zz \rangle \mapsto h$ and $\ds \rho : \langle h, \zz \rangle \mapsto \zz.$ Note that $\pi$ is a group homomorphism. We define \eq{constants}{ D_K := \sup \left\lbrace \|h\| : h \in \pi(K)  \right\rbrace, \ \ \ \ \ \ \ \ \ \ \ \ \ \ \ E_K := \sup \left\lbrace \nu(\zz) : \zz \in \rho(K) \right\rbrace. } Note that $E_K = 0$ if and only if $K \subseteq \SL.$ Note that $D_K \geq 1.$ Let $\ds a = a_\psi$ and $\ds b= b_\psi$ be as in Definition~\ref{basicapdefns}. Set $\ds {k} := \min\left\lbrace j \in \Z_{\geq 0} : a^j \geq D_K \right\rbrace.$ Set $C_K := b^{-{k}}.$ Note that $C_K \geq 1.$ Appealing to Lemma \ref{sublinearnew}, we let $F_K \in \R_{\geq 1}$ be a constant for which the following is true: for each $x \in [0, E_K]$ and each $y \in (E_K, \infty),$ we have $\ds \psi(y-x) \les F_K \psi(y).$ Set $\ds J_K := C_K F_K.$ 

   
Let $\ds \langle h_1,\zz_1 \rangle \in K$ be arbitrary. Let $R$ be any real number with $R > 2D_KE _K.$ Let $\xx$ be any element of $\R^n$ with $\ds 2 D_K E_K < \nu(\xx) \leq R.$ Then \[ \norm{h_1^{-1}}\nu(h_1\xx) \geq \nu(\xx) > 2D_K E_K \geq 2 \norm{h_1^{-1}} E_K, \] whence $\ds \nu(h_1\xx) > 2E_K.$ It follows \[ \nu(h_1\xx + \zz_1) \geq \nu(h_1\xx) - \nu(\zz_1) > 2E_K - \nu(\zz_1) \geq E_K. \]  Suppose further that $\ds f(h_1\xx + \zz_1) \les \psi\big( \nu(\xx)\big).$ Then 
\begin{align*}
f(h_1\xx + \zz_1) &\les \psi\big(\nu(\xx)\big) \les \psi\left( \frac{\nu(h_1\xx)}{\|h_1\|}\right) \les b^{-{k}} \psi\left( a^{k} \frac{\nu(h_1\xx)}{\|h_1\|}\right) \les C_K \psi\big(\nu(h_1\xx)\big) \\
&\les C_K \psi\big( \nu(h_1\xx + \zz_1) - \nu(\zz_1)\big) 
\les C_K F_K \psi\big( \nu(h_1\xx + \zz_1) \big) 
= J_K \psi\big( \nu(h_1\xx + \zz_1) \big).
\end{align*}
Finally, we note that $\ds \nu(h_1\xx + \zz_1) \leq D_K R + E_K.$ We have therefore shown 
\smallskip
\eq{supinc}{ \langle h_1, \zz_1 \rangle
\left\lbrace \ttt \in A_{ f \circ \langle h_1, \zz_1 \rangle, \psi, \nu } : 2 D_K E_K < \nu(\ttt) \leq R \right\rbrace 
 \subseteq \left\lbrace \ttt \in A_{f, J_K \psi, \nu} : E_K < \nu(\ttt) \leq D_K R + E_K \right\rbrace.} 

\smallskip

By Lemma \ref{normimmmeas}, we have $\ds m\left(\left\lbrace \ttt \in A_{f, J_K \psi, \nu} : \nu(\ttt) > E_K \right\rbrace\right) = \infty.$ By using \equ{supinc} and then applying Theorem \ref{genericcounting}(ii), it follows that for $\mu_G$-almost every $ \langle h,\zz\rangle \in K$ and any $\varepsilon \in \R_{>0}$ there exists some $\ds T_{ \langle h,\zz\rangle} \in \R_{>0}$ such that for every 
$T \geq T_{ \langle h,\zz\rangle},$ we have
\begin{align*}
&\frac{\card\left\lbrace \vv \in {{\mathcal{P}}} : (f \circ  \langle h,\zz\rangle )(\vv) \les \psi\big( \nu(\vv) \big) \ \text{and} \ 2 D_K E_K < \nu(\vv) \leq T \right\rbrace}{m\left(\left\lbrace \ttt \in \R^n : f(\ttt) \les J_K\psi\big(\nu(\ttt)\big) \ \text{and} \ {E_K} < \nu(\ttt) \leq D_K T + E_K \right\rbrace\right)} \\
&\leq \frac{\card\left\lbrace \ww \in  \langle h,\zz\rangle {{\mathcal{P}}} : f(\ww) \les J_K\psi\big(\nu(\ww)\big) \ \text{and} \ {E_K} < \nu(\ww) \leq D_K T + E_K \right\rbrace}{m\left(\left\lbrace \ttt \in \R^n : f(\ttt) \les J_K\psi(\nu\big(\ttt)\big) \ \text{and} \ {E_K} < \nu(\ttt) \leq D_K T + E_K \right\rbrace\right)} \\
&< {{c}_{{\mathcal{P}}}} + \varepsilon.
\end{align*}
It follows that   \eqref{limsup} holds for $\mu_G$-almost every $ \langle h,\zz\rangle \in K$.

\smallskip

Now let $\ds \langle h_2,\zz_2 \rangle \in K$ be arbitrary. Let $R'$ be any real number with $R' > 2E_K.$ By an argument similar to the one given for \equ{supinc}, one can show 
\smallskip
\eq{infinc}{ \langle h_2, \zz_2 \rangle^{-1}
\left\lbrace \ttt \in  A_{f, J_K^{-1}\psi, \nu} : 2 E_K <\nu(\ttt) \leq R' \right\rbrace
 \subseteq \left\lbrace \ttt \in A_{f \circ \langle h_2, \zz_2 \rangle, \psi, \nu} : {E_K D_K^{-1}} < \nu(\ttt) \leq D_K(E_K + R') \right\rbrace.}

\smallskip

By Lemma \ref{normimmmeas}, we have $\ds m\left(\left\lbrace \ttt \in  A_{f, J_K^{-1}\psi, \nu} : \nu(\ttt) > 2 E_K \right\rbrace\right) = \infty.$ By using \equ{infinc} and then applying Theorem \ref{genericcounting}(ii), it follows that for $\mu_G$-almost every $\langle h,\zz\rangle \in K$ and any $\varepsilon \in \R_{>0}$ there exists some $\ds T_{ \langle h,\zz\rangle}' \in \R_{>0}$ such that for every real $T \geq T_{ \langle h,\zz\rangle}',$ we have
\begin{align*}
&\frac{\card\left\lbrace \vv \in {{\mathcal{P}}} : (f \circ  \langle h,\zz\rangle )(\vv) \les \psi\big( \nu(\vv)\big) \ \text{and} \ {E_K D_K^{-1}} < \nu(\vv) \leq D_K(E_K + T) \right\rbrace}{m\left(\left\lbrace \ttt \in \R^n : f(\ttt) \les J_K^{-1}\psi\big(\nu(\ttt)\big) \ \text{and} \ 2E_K < \nu(\ttt) \leq T \right\rbrace\right)} \\
&\geq \frac{\card\left\lbrace \ww \in  \langle h,\zz\rangle {{\mathcal{P}}} : f(\ww) \les J_K^{-1}\psi\big(\nu(\ww)\big) \ \text{and} \ 2E_K < \nu(\ww) \leq T \right\rbrace}{m\left(\left\lbrace \ttt \in \R^n : f(\ttt) \les J_K^{-1}\psi\big(\nu(\ttt)\big) \ \text{and} \ 2E_K < \nu(\ttt) \leq T \right\rbrace\right)} \\
&> {{c}_{{\mathcal{P}}}} - \varepsilon.
\end{align*}
Thus, \eqref{liminf} holds for  $\mu_G$-almost every $ \langle h,\zz\rangle \in K$.
%
%
The final statement of (ii) now follows from the $\sigma$-compactness of $G$ and an application of Lemma \ref{normimmmeas}. 
\end{proof} 

\ignore{\begin{rmk}\label{notsubhomog} \rm
{The assumption that $f = (f_1, \dots , f_\ell) : \R^n \to \R^\ell$ is subhomogeneous in the foregoing theorem may be replaced with the following assumption:}  

\smallskip
{\noindent \sl {The function $f$ is Borel measurable, and for any norms $\eta$ and $\nu$ on $\R^n$ and any $s \in \R_{>0},$ the Lebesgue measure of $A_{f, s\psi, \eta}$ is finite if and only if that of $A_{f, \psi, \nu}$ is finite.} } 
\smallskip

{The same proof then works essentially verbatim:~one simply has to cite the above aforementioned assumption instead of Lemma \ref{normimmmeas}. In other words, subhomogeneity is necessary for the preceding proof only insofar as it implies the conclusion of Lemma \ref{normimmmeas} (and Borel measurability).  }
\end{rmk}}

We now prepare to prove our results on uniform approximation. We first prove a lemma similar to Lemma \ref{normimmmeas}.  

\smallskip
\begin{lem}\label{uniformscaling}
 {Let $\ds f 
: \R^n \to \R^{\ell}$ be subhomogeneous, with $\ds d = d_f \in \R_{>0}$ as in Definition \ref{basicapdefns}.}
\begin{itemize} 
\item[\rm (i)] {Let $\nu$ be an arbitrary norm on $\R^n,$ let $t \in (0, 1),$ $T \in \R_{>0},$ and ${\pmb\varepsilon} \in \left(\R_{>0}\right)^\ell.$ Then $ {t} B_{f, {\pmb\varepsilon}, \nu, T} \subseteq B_{f, {t}^d{\pmb\varepsilon}, \nu, {t} T}$.}

\item[\rm (ii)] {Let $\nu$ and $\eta$ be arbitrary norms on $\R^n.$ Then there exists $\ds C^* = C_{\nu, \eta}^* \in \R_{\geq 1}$ such that for each $\ds C \in [C^*, \infty),$ each $T \in \R_{>0},$ and each ${\pmb\varepsilon} \in \left(\R_{>0}\right)^\ell,$ we have \[ C^{-1} B_{f, {\pmb\varepsilon}, \nu, T} \subseteq B_{f, C^{-d}{\pmb\varepsilon}, \eta, T} \subseteq B_{f, {\pmb\varepsilon}, \eta, T} \subseteq C\, B_{f, C^{-d}{\pmb\varepsilon}, \nu, T} \subseteq C\, B_{f, {\pmb\varepsilon}, \nu, T}. \]} 
\end{itemize}
\end{lem}

\begin{proof}
Let $\xx \in {t} B_{f, {\pmb\varepsilon}, \nu, T}.$ Then, since  $t^{-1}\xx \in  B_{f, {\pmb\varepsilon}, \nu, T},$
we have 
$$
\nu(\xx) 
= 
{t}\nu(t^{-1}\xx) \leq {t} T\quad  \text{and}\quad |f(\xx)| 
\les {t}^d |f(t^{-1}\xx)| \les {t}^d {\pmb\varepsilon}.$$ Hence, $\xx \in B_{f, {t}^d{\pmb\varepsilon}, \nu, {t} T}$, which proves (i). 

\smallskip

For (ii), fix $\ds C^* = C_{\nu, \eta}^* \in \R_{\geq 1}$ for which $\ds \left(C^*\right)^{-1}\eta(\cdot) \leq \nu(\cdot) \leq C^* \eta(\cdot).$ Fix any $\ds C \in [C^*, \infty)$ and 
let $\ds \xx \in C^{-1} B_{f, {\pmb\varepsilon}, \nu, T}$. 
We have $\ds 
\eta(\xx) 
 \leq \nu(C\xx) \leq T.$ Moreover, one has 
 \[ 
 |f(\xx)| 
 \les C^{-d} |f(C\xx)| \les C^{-d} {\pmb\varepsilon} \les {\pmb\varepsilon}
 \quad\Longrightarrow\quad
 C^{-1} B_{f, {\pmb\varepsilon}, \nu, T} \subseteq B_{f, C^{-d}{\pmb\varepsilon}, \eta, T} \subseteq B_{f, {\pmb\varepsilon}, \eta, T}.\]
 
Interchanging $\nu$ and $\eta$ and then arguing similarly, one obtains
$$ B_{f, {\pmb\varepsilon}, \eta, T} \subseteq C\, B_{f, C^{-d}{\pmb\varepsilon}, \nu, T} \subseteq C\, B_{f, {\pmb\varepsilon}, \nu, T}.$$ This completes the proof. \end{proof}

Let us now introduce some definitions and then prove another lemma. 


\begin{defn}\label{discreteapproxdefn}
Let $\ds f 
: \R^n \to \R^\ell$ and $\psi 
: \R_{\geq 0} \to \left(\R_{>0}\right)^\ell$ be arbitrary maps. Let $\nu$ be an arbitrary norm on $\R^n,$ and let $\mathcal{P}$ be an arbitrary subset of $\Z^n.$ Let $\ds t_\bullet = \left(t_k\right)_{k \in \Z_{\geq 1}}$ be any strictly increasing sequence of elements of $\R_{>0}$ with $\ds \lim_{k \to \infty} t_k = \infty.$ We say that $f$ is $\ds t_\bullet$-\textsl{uniformly} $\left( \psi, \nu, \mathcal{P} \right)$-approximable if 
 {$B_{f, \psi(t_k), \nu, t_k} \cap \mathcal{P}  \neq \varnothing$ for each sufficiently large $k \in \Z_{\geq 1}$}.
\end{defn}


\begin{defn}\label{nicesequence}
Let $\ds t_\bullet = \left(t_k\right)_{k \in \Z_{\geq 1}}$ be any strictly increasing sequence of elements of $\R_{>0}$ with $\ds \lim_{k \to \infty} t_k = \infty.$ We say that $t_\bullet$ is \textsl{quasi-geometric} if, in addition to the preceding, the set $\ds \left\lbrace \frac{t_{k+1}}{t_k} : k \in \Z_{\geq 1} \right\rbrace$ is bounded. 
\end{defn}


\begin{thm}\label{UniformMain}
Let $G$ be a closed subgroup of $\ASL$, let $\Gamma$ be as in \equ{gamma}, and let ${{\mathcal{P}}}$ be a $\Gamma$-invariant subset of $\Z^n.$ Let $\ds t_\bullet = \left(t_k\right)_{k \in \Z_{\geq 1}}$ be any strictly increasing sequence of elements of $\R_{>0}$ with $\ds \lim_{k \to \infty} t_k = \infty.$ Suppose $G$ is of $\mathcal{P}$-Siegel type, and suppose we are given $r \in \R_{> 1}$ for which $G$ is of $\ds \left(\mathcal{P}, r\right)$-Rogers type. Let $\ds \psi 
: \R_{\geq 0} \to \left(\R_{>0}\right)^{\ell}$ be Borel measurable, and let $\ds f 
: \R^n \to \R^{\ell}$ be subhomogeneous. Suppose also that there exists some norm $\eta$ on $\R^n$ for which \eq{conv}{\ds \sum_{k=1}^{\infty}  m (B_{f, \psi\left(t_k\right), \eta, t_k})^{1-r} < \infty.} {Let $\nu$ be an arbitrary norm on $\R^n$.} We then have the following. 
\begin{itemize}
    \item[\rm (i)] For 
    almost every $g \in G$ the function $f \circ g$ is $\ds t_\bullet$-uniformly $\ds (\psi, \nu, \mathcal{P})$-approximable. 
    \item[\rm (ii)] Suppose further that 
   $\ds \psi 
    $ is nonincreasing and regular, and the sequence $t_\bullet$ is quasi-geometric. Then for 
    almost every $g \in G$ the function $f \circ g$ is uniformly $\ds (\psi, \nu, \mathcal{P})$-approximable. 
\end{itemize}
\end{thm}
\begin{proof}  
\begin{itemize}
    \item[\rm (i)]  
  {Let $K$ be a  nonempty compact subset of $G$; as in the proof of Theorem \ref{mainapproxthmnew}, assume without loss of generality that $K = K^{-1}$. Define the constants $D_K$, $E_K$ by \equ{constants}. Lemma \ref{uniformscaling} and \equ{conv} imply that the series $\ds \sum_{k=1}^\infty  m (B_{f, \psi\left(t_k\right), \nu, 
  {t_k}/{(2D_K)}}) ^{1-r} 
  $ converges. Applying Theorem \ref{genericcounting}(iii), we obtain the following: For almost every $g \in G$ there exists $M_{g} \in{\Z_{\ge 1}}$ such that
  for each $k \in \Z$ with $k \geq M_g$  there exists some $\vv_k \in \mathcal{P}$ with 
  \eq{boundvk}{\ds 
   \nu\left(g \vv_k\right) \leq \frac{t_k}{2D_K}\quad\text{and}\quad\ds |f(g\vv_k)| \les \psi\left(t_k\right).}

{For each such $g \in G,$ we assume without loss of generality that for each $k \in \Z$ with $k \geq M_g,$ we have $\ds t_k > 2E_K$.}} {If, in addition, $g \in K$,
then for any} 
$\vv_k$ as above  it now follows from  \equ{boundvk} that we have $\nu( \vv_k) \leq {t_k}$. {Indeed,} if not, then we write  $g^{-1} = \langle h,\zz\rangle$, as in  \equ{affine}, and note that \begin{align*}\nu\big(h (g\vv_k) + \zz\big) = \nu( \vv_k)  > {t_k}\quad&\Longrightarrow\quad\nu\big(h (g\vv_k)\big)  > {t_k} - \nu(\zz) \ge {t_k} - E_K \ge \frac{t_k}{2}\\\ &\Longrightarrow\quad\nu( g\vv_k)  > 
\frac{t_k}{2D_K},\end{align*}  {which is a contradiction.}  
Therefore, for $\mu_G$-almost every $g \in K$ and with $M_g$ as above, it follows that for every $k \in \Z$ with $k \geq M_g$ there exists $\vv_k \in  \mathcal{P}$ with $\ds |(f \circ g)(\vv_k)| \les \psi\left(t_k\right)$ and $\ds 
\nu(\vv_k) \leq 
t_k.$ Hence,
for $\mu_G$-almost every $g \in {K}$, the function $f \circ g$ is $\ds t_\bullet$-uniformly $\ds (\psi, \nu, \mathcal{P})$-approximable. Since $G$ is $\sigma$-compact, the same holds for almost every $g \in G$.
 \smallskip
 
\item[\rm (ii)] Let $\ds a = a_\psi$ and $b = b_\psi$ be as in Definition \ref{basicapdefns}.  {Fix $j \in \Z_{\geq 1}$ for which $$ \sup\left\lbrace {t_{k+1} /t_k} : k \in \Z_{\geq 1} \right\rbrace < a^j\quad\text{and}\quad \sum_{k=1}^{\infty}  m (B_{f, b^j\psi\left(t_k\right), \eta, t_k}) ^{1-r} < \infty;$$ notice that this is indeed possible in light of Definition \ref{nicesequence}, Lemma \ref{uniformscaling}, and \equ{conv}. } By statement (i), we know that for almost every $g \in G,$ the function $f \circ g$ is $t_\bullet$-uniformly $\left(b^j\psi, \nu, \mathcal{P}\right)$-approximable. Now let $\ds h 
: \R^n \to \R^\ell$ be any function that is $t_\bullet$-uniformly $\left(b^j\psi, \nu, \mathcal{P}\right)$-approximable. Fix ${M} \in \Z_{\geq 1}$ such that for each $k \in \Z_{\geq {M} }$ the set $\ds B_{h, b^j\psi(t_k), \nu, t_k} \cap \mathcal{P}$ is nonempty. Let $\ds T \in \left( t_{{M} +2}, \infty\right)$ be arbitrary. Then there exists $i \in \Z_{\geq {M} +2}$ for which $\ds t_i \leq T < t_{i+1}.$ Note that there exists $\vv \in \mathcal{P}$ with $
    \nu(\vv) \leq t_i$ and $\ds |h(\vv)| \les b^j \psi(t_i).$ We then have $
     \nu(\vv) \leq t_i \leq T$ and $\ds |h(\vv)| \les b^j \psi(t_i) \les b^j b^{-j} \psi\left(a^j t_i \right) = \psi\left(a^j t_i \right) \les \psi\left(t_{i+1}\right) \les \psi(T).$ \end{itemize}\end{proof}  

\begin{proof}[Proof of Theorem \ref{simplemainapproxthm}] Theorem \ref{simplemainapproxthm} is now an immediate consequence of Theorems \ref{RogersEg}(ii), \ref{SL2}, \ref{mainapproxthmnew}, and \ref{UniformMain}(ii) {with  {$\mathcal{P} = \Z^n_{\ne0}$ and} $\ds t_\bullet = 
(2^k)_{k \in \Z_{\geq 1}}$}. 
\end{proof}

\begin{rmk} \rm
Denote by $Z_n$ the group of scalar $n\times n$ matrices (that is, the center of $\GL$). For any $G$ as in Theorem \ref{genericcounting} set $\widetilde{G} := G\times Z_n$; we have, for example, $\widetilde{\SL} = \GL$ and $\widetilde{\ASL} = \AGL$. It is then clear from the Fubini--Tonelli Theorem that every result of this section that was established for $G$ also holds, \textit{mutatis mutandis}, for $\widetilde{G}.$ The same remark applies to the corollaries derived in the next section. 
Alternatively, the $\mathrm{GL}_n(\R)$ analogue of our results follows easily from the corollary to Theorems 1 and 2 in \cite{Schmidt}, via an application of Lemma \ref{normimmmeas}. 
\end{rmk}

\begin{rmk} \rm 
{Let us note that the null and conull subsets of $G$ in each part of Theorem \ref{mainapproxthmnew} and in Theorem \ref{UniformMain}(ii) may be chosen independently of the norm. This is an immediate consequence of the facts that all norms on $\R^n$ are equivalent,  that any {positive} multiple of a norm is a norm, that $\Z_{\geq 1}$ is countable and unbounded, and that $\psi = \left( \psi_1, \dots , \psi_\ell \right) : \R_{\geq 0} \to \left(\R_{>0}\right)^\ell$ is assumed to be nonincreasing in the results that were just mentioned.} 
\end{rmk}

In the following section, we apply Theorems \ref{mainapproxthmnew} and \ref{UniformMain} to investigate the orbits of several specific subhomogeneous functions $f.$ We do so by performing several volume calculations.

\section{Examples and volume calculations}\label{examplesvol}  
{Let us state the conventions that will be in force throughout this section.} We shall let $G$ denote a closed subgroup of $\ASL$, {$n\in\Z_{\ge 2}$,} and $\mathcal{P}$ denote a $\Gamma$-invariant subset of $\Z^n,$ where $\Gamma$ is as \equ{gamma}. We shall assume $G$ is of $\mathcal{P}$-Siegel type and that there exists ${r} \in \R_{>1}$ for which $G$ is of $\ds \left(\mathcal{P}, {r}\right)$-Rogers type. 
We let $\eta$ denote an arbitrary norm on $\R^n,$ and we let $\psi : \R_{\geq 0} \to \R_{>0}$ denote an arbitrary nonincreasing and regular function.
%

\begin{cor}{\label{quadforms}}
Let $d \in \R_{\geq 1},$ and fix any $p, q \in \Z_{\geq 1}$ with $p+q = n.$ Define $f : \R^n \to \R$ by \equ{special}. Define the norm $\nu$ on $\R^n = \R^p \times \R^{q}$ by \eq{norm}{\nu\big( (\textbf{x},\textbf{y})\big) := \max\left(\|\xx\|_d,\|\yy\|_d\right),} where $\|\cdot\|_d$ denotes the $\ell^d$ norm on each of the spaces $\R^p$ and $\R^{q}.$ 
For each $k \in \{p, q\}$, we let ${v}_k$ denote the volume of the unit ball in $\R^k$ and let ${v}_k'$ denote the  volume of the  unit sphere in $\R^k$ {\rm (}each taken with respect to the  $\ell^d$ norm on $\R^k${\rm )}. Then the following hold. 

\begin{itemize}
    \item[\rm (i)] There exists some $M \in {\R_{\ge 1}}$ such that 
    for any $T\ge S\ge M$ 
    we have 
    \[
    \begin{aligned}
     &m\left(A_{f, \psi, \nu} \cap \{\xx \in \R^n : S \leq \nu(\xx) \leq T \} \right)  \\  & =  \int_S^T {z}^{n-1}\left[{v}_p {v}_q'  \left( 1 - \left(1 - \frac{\psi\left({z}\right)}{{z}^d}\right)^{p/d} \right) + {v}_q {v}_p'  \left( 1 - \left(1 - \frac{\psi\left({z}\right)}{{z}^d}\right)^{q/d} \right)\right]d{z}.\end{aligned}\]
    \item[\rm (ii)] For almost every $g \in G$ the function $f\circ g$ is \textup{(}resp., is not\textup{)} $(\psi, \eta, {{\mathcal{P}}})$-approximable if the integral \eq{correctone}{\int_1^{\infty} \psi({z}) {z}^{n-(d+1)} \,d{z}} is infinite \textup{(}resp.,  finite\textup{)}.
    
\item[\rm (iii)] 
Suppose that the series
\eq{iii}{ \begin{cases}\ds\sum_{k = 1}^\infty
 { \left( k\psi(2^k) \right)^{1-r}}  &  \text{ if }d = n  \\ \ds\sum_{k = 1}^\infty  
 {  \left( 2^{(n-d)k}\psi(2^k)  \right)^{1-r}} & \text{ if }d \ne n \end{cases}} \noindent  converges. Then for almost every $g \in G$ the function $f \circ g$ is uniformly $(\psi, \eta, \mathcal{P})$-approximable.






\end{itemize}   
\end{cor}  

For the next example, we consider the space of products of $n$ linearly independent linear forms on $\R^n$.  In what follows, for each $j \in \Z_{\geq 0},$ we write $\log^{j}$ to denote the function $\R_{>0} \to \R$ given by $z \mapsto \left(\log z \right)^{j}.$ 

\begin{cor}\label{product} 
Define $f : \R^n \to \R$ by $f(x_1, \dots , x_n) := x_1 \cdots x_n$. Let $\nu$ denote the maximum norm on $\R^n.$ 
Then there exists  $M \in \R_{\geq 1}$ such that: 
\begin{itemize}
    \item[\rm (i)] 
    For any ${T}\ge {S} \ge M$     we have 
  \eq{measureexact}{ m\left(A_{f, \psi, \nu} \cap \{\xx \in \R^n : {S} \leq\nu(\xx) \leq {T} \} \right) = 2^n \, n \int_{S}^{T} \frac{\psi({z})}{{z}}
 { \left[ \sum_{i=0}^{n-2} \frac{1}{i!} \log^i\left(\frac{{z}^n}{\psi({z})}\right)\right]}
   \,d{z}.}
    
    \item[\rm (ii)] For almost every $g \in G$ the function $f\circ g$ is \textup{(}resp., is not\textup{)} $(\psi, \eta, {{\mathcal{P}}})$-approximable if the integral \[\int_1^{\infty} \frac{\psi({z})}{{z}}   \log^{n-2} \left(\frac{{z}^n}{\psi({z})} \right) \,d{z}\] is infinite \textup{(}resp.,  finite\textup{)}.  
    
\item[\rm (iii)] 
{Suppose} that the series
$${ \begin{cases}\ds\sum_{k = 1}^\infty
{ \big( k\psi(2^k) \big)^{1-r}}   &
\text{ if }n=2 \\ \ds\sum_{k = 
{1}}^\infty 
{  \left( 
\psi(2^k)
\log^{n-1}\big(2^k \psi(2^k)^{-
{1}/{n}}\big)
 \right)
^{1-r}}    &
\text{ if }n > 2
\end{cases}}$$
  %
converges. Then for almost every $g \in G$ the function $f \circ g$ is uniformly $(\psi, \eta, \mathcal{P})$-approximable.

\ignore{\smallskip
    
\item[\rm (iv)] If $s\in \R_{\geq 0},$ then it follows that for almost every $g \in G,$ the function $f\circ g$ is \textup{(}resp., is not\textup{)} $(\varphi_s, \eta, {{\mathcal{P}}})$-approximable if $s = 0$ \ \textup{(}resp., $s > 0$\textup{)}.

\smallskip
    
\item[\rm (v)] If $s\in \R_{\geq 0},$ then it follows that for almost every $g \in G,$ the function $f\circ g$ is \textup{(}resp., is not\textup{)} uniformly $(\varphi_s, \eta, {{\mathcal{P}}})$-approximable if $s = 0$ \ \textup{(}resp., $s > 0$\textup{)} }
\end{itemize}       
\end{cor}    

 {Notice that Corollary 4.2(ii) is 
similar to \cite[Theorem 1.11]{KM}; indeed, \cite[Theorem 1.11]{KM} implies Corollary 4.2(ii) in the special case that $G = \SL,$ $\mathcal{P} = \Z^n_{\neq 0}$ and $\eta$ is the maximum norm on $\R^n$.
Whereas the proof of Kleinbock\textendash Margulis in \cite{KM} relied on the Dani correspondence and the exponential mixing of the $\SL$-action on $\SL/{\rm SL}_n(\Z),$ our proof only uses the expectation and variance formulae of the Siegel transforms.}  
\smallskip 

The next example is of interest because of its relation to the Khintchine\textendash Groshev Theorem; {see Remark \ref{kgt}}.    

\begin{cor}\label{khintchine}
Let $\ell \in \{1, \dots , n - 1 \}$ and $\mathbf{a} = (a_1, \dots , a_\ell) \in \left(\R_{>0}\right)^\ell$ be given. Define $f: \R^n \to \R$ by \[f(x_1, \dots , x_n) := \max\left(|x_1|^{a_1}, \dots , |x_\ell|^{a_\ell}\right).\] Set $\ds a := \sum_{i=1}^\ell  a_i^{-1}.$ Let $\nu$ denote the maximum norm on $\R^n.$ Then:  
\begin{itemize}
    \item[\rm (i)] There exists some $M \in \R_{>0}$ such that for any ${T}\ge {S} \ge M$
    we have \[ m\left(A_{f, \psi, \nu} \cap \{\xx \in \R^n : {S} \leq\nu(\xx) \leq {T} \} \right) = 2^n (n-\ell) \int_{S}^{T} \psi({z})^a {z}^{n-(\ell+1)} \ d{z}. \]
    
    \item[\rm (ii)] For almost every $g \in G$ the function $f\circ g$ is \textup{(}resp., is not\textup{)} $(\psi, \eta, {{\mathcal{P}}})$-approximable if the integral 
    $${ \int_1^{\infty} \psi({z})^a {z}^{n - (\ell+1)} \,d{z}}$$ is infinite \textup{(}resp.,  finite\textup{)}.   
    
    \item[\rm (iii)] Suppose that $\ds \sum_{k = 1}^\infty 
    {  \left( 2^{(n-\ell)k}\psi(2^k) ^{{a}} \right)^{1-r}}
    $ converges. Then for almost every $g \in G$ the function $f \circ g$ is uniformly $(\psi, \eta, \mathcal{P})$-approximable.
    
\ignore{    \smallskip
    
    \item[\rm (iv)] Given any $s \in \R_{\geq 0},$ it follows that for almost every $g \in G,$ the function $f\circ g$ is \textup{(}resp., is not\textup{)} $(\varphi_s, \eta, {{\mathcal{P}}})$-approximable if $\ds s \leq \frac{n-k}{a}$ \ $\ds \left(\textsl{resp.,} \  s > \frac{n-k}{a}\right).$ 
    
    \item[\rm (v)] Given any $s \in \R_{\geq 0},$ it follows that for almost every $g \in G,$ the function $f\circ g$ is \textup{(}resp., is not\textup{)} uniformly $(\varphi_s, \eta, {{\mathcal{P}}})$-approximable if $\ds s < \frac{n-k}{a}$ \ $\ds \left(\textsl{resp.,} \  s > \frac{n-k}{a}\right).$   }     
\end{itemize}  
\end{cor}

Before proving these corollaries, let us make a few remarks. 
\begin{rmk}\label{explanations} \rm
\begin{itemize}
  \item[\rm (i)]  Corollary \ref{spequadforms} is clearly a special case of Corollary \ref{quadforms}.
  \smallskip
  
  \item[\rm (ii)]  {The $n=2$ case of Corollary \ref{product} coincides with the $n=d=2$ case of Corollary \ref{quadforms}; however, the two volume formulas are slightly different due to the difference in the  choice of norms.}
 \smallskip
  
   \item[\rm (iii)]  
   In each corollary, part (i) does not require the regularity of $\psi$ and is valid for any nonincreasing function $\psi : \R_{\geq 0} \to \R_{>0}.$
 \smallskip
  
   \item[\rm (iv)]  
   In each of the corollaries, parts (ii) and (iii) may be used to calculate the critical exponents for asymptotic and uniform approximability, respectively: that is, the supremum of the set of all $s \in \R_{\geq 0}$ such that almost every element in the $G$-orbit of $f$ is $(\varphi_s,\nu,{\mathcal P})$-approximable or uniformly $(\varphi_s,\nu,{\mathcal P})$-approximable, respectively, {where $\varphi_s$ is as in \eqref{psis}}. In each corollary, one readily obtains that this supremum, {if finite}, is actually a maximum in the case of asymptotic approximation, and also that the critical exponents for asymptotic and uniform approximability coincide.  
 \smallskip
  
 \item[\rm (v)]  
 Instead of using the sequence $\left(2^k\right)
 $ in part (iii) of each corollary, one may instead use any quasi-geometric sequence $\ds 
  (t_k)
  $ that in addition is  {\sl lacunary}; that is, 
 $\ds \inf \left\lbrace \frac{t_{k+1}}{t_{k}} : k \in \Z_{\geq 1} \right\rbrace > 1.$ 
In fact, it is not hard to prove that if $F : \R_{>0} \to \R_{>0}$ is any Borel measurable function that satisfies some additional mild conditions, then the following are equivalent: 

\begin{itemize}
    \item[(a)] There exists a quasi-geometric sequence $\ds 
    \left(t_k\right)_{k \in \Z_{\geq 1}}$ for which $\ds \sum_{k=1}^{\infty} F(t_k) < \infty.$
    
    \item[(b)] 
    $\ds \int_1^{\infty} \frac{F(x)}{x} \ dx < \infty.$
    
    \item[(c)]  
     $\ds \sum_{k=1}^{\infty} F(t_k) < \infty$ for any quasi-geometric and lacunary sequence $\left(t_k\right)_{k \in \Z_{\geq 1}}$. 
\end{itemize}
\end{itemize}
\end{rmk}    

\noindent {Let us now proceed to prove the corollaries.} 

\begin{proof}[Proof of Corollary \ref{quadforms}] 
Recall that the function $f : \R^n \to \R$ is defined by \equ{special}; that is, $f\big( (\textbf{x},\textbf{y})\big) = \|\xx\|_d^d - \|\yy\|_d^d$. 
For any 
{$T\ge S \ge 0$} and with the norm $\nu$ on $\R^n = \R^p \times \R^{q}$ given by 
\equ{norm},
 we define
\begin{align*}
A_{{S}}^{{T}} &:= \left\lbrace(\xx, \yy) \in 
A_{f, \psi, \nu} 
 : 
 {S} \leq \nu\left( (\xx, \yy) \right) \leq {T} \right\rbrace, \\
_{\xx}A_{{S}}^{{T}} &:= A_{{S}}^{{T}} \cap \left\lbrace(\xx, \yy) \in \R^p \times \R^q : \|\yy\|_d \leq \|\xx\|_d \right\rbrace, \\
_{\yy}A_{{S}}^{{T}} &:= A_{{S}}^{{T}} \cap \left\lbrace(\xx, \yy) \in \R^p \times \R^q : \|\xx\|_d \leq \|\yy\|_d \right\rbrace.
\end{align*}    

Since the function $\R_{>0} \to \R$ given by ${z} \mapsto {z}^d - \psi({z})$ is strictly increasing and unbounded from above, there exists 
$M \in {\R_{\ge 1}}$ such that for each ${z} \in [M, \infty),$ we have ${z}^d - \psi({z}) > 0.$ Now 
{suppose that $T\ge S \ge M$.} Then
\begin{align*}
m\left(_{\xx}A_{{S}}^{{T}}\right) &= m\left(\left\lbrace (\xx, \yy) \in \R^p \times \R^q : \sqrt[d]{\|\xx\|_d^d - \psi\left(\|\xx\|_d\right)} \leq \|\yy\|_d \leq \|\xx\|_d \ \ \text{and} \ \ {S} \leq \|\xx\|_d \leq {T} \right\rbrace\right) \\
&= {v}_q \int_{\{\xx \in \R^p : {S} \leq \|\xx\|_d \leq {T}\}}  \left( \|\xx\|_d^q - \left(\|\xx\|_d^d - \psi\left(\|\xx\|_d\right)\right)^{{q}/{d}} \right) \, d{\xx} \\
&={v}_q {v}_p'  \int_{S}^{T} {z}^{p-1} \left( {z}^q - \left({z}^d - \psi\left({z}\right)\right)^{{q}/{d}} \right) \, d{z} \\
&= {v}_q {v}_p' \int_{S}^{T} {z}^{n-1} \left( 1 - \left(1 - \frac{\psi\left({z}\right)}{{z}^d}\right)^{{q}/{d}} \right) \, d{z}.
\end{align*}
By symmetry, we have \[ m\left(_{\yy}A_{{S}}^{{T}}\right) = {v}_p {v}_q' \int_{S}^{T} {z}^{n-1} \left( 1 - \left(1 - \frac{\psi\left({z}\right)}{{z}^d}\right)^{{p}/{d}} \right) \, d{z},\]
which completes the proof of (i). 
\smallskip

By using the Taylor expansions of the functions $x\mapsto 1 - \left(1 - x\right)^{{p}/{d}}$ and $x\mapsto 1 - \left(1 - x\right)^{{q}/{d}}$ around $0$, it is easy to see that 
there exist some $S \in {\R_{\ge M}}$ 
and $C_1, C_2 \in \R_{>0}$ such that 
{\eq{taylor}{z \in [S, \infty)\ \Longrightarrow\ 
\left\{1 - \left(1 - \frac{\psi\left({z}\right)}{{z}^d}\right)^{{p}/{d}},
1 - \left(1 - \frac{\psi\left({z}\right)}{{z}^d}\right)^{{q}/{d}}\right\}
\subset \left[C_1\frac{\psi\left({z}\right)}{{z}^d}, C_2\frac{\psi\left({z}\right)}{{z}^d}\right].}}

\ignore{For each $r \in [{S}, \infty),$ we have $\ds 0 < \frac{\psi(r)}{r^d} \leq  \frac{\psi({S})}{R^d} \leq \frac{\psi(r_0)}{r_0^d} < 1.$ It follows that for any $\alpha \in \R_{>0},$ we have
\begin{align*}
&\int_{S}^{T} r^{n-1} \left( 1 - \left(1 - \frac{\psi\left(r\right)}{r^d}\right)^\alpha \right) dr \\
&= \int_{S}^{T} r^{n-1} \left( \sum_{k=1}^\infty \binom{\alpha}{k} (-1)^{k+1} \left(\frac{\psi(r)}{r^d}\right)^k \right) dr \\
&= \int_{S}^{T} \psi(r) r^{n-(d+1)} \left( \sum_{k=1}^\infty \binom{\alpha}{k} (-1)^{k+1} \left(\psi(r)\right)^{k-1} r^{-d(k-1)} \right) dr.
\end{align*}

It is well-known that the set \[ \left\lbrace \absval{\binom{\alpha}{k}} : k \in \Z_{\geq 0} \right\rbrace \] is bounded.

Let $\varepsilon > 0$ be given. For any sufficiently large $r \geq r_0,$ we then have \[ \absval{\sum_{k=2}^\infty \binom{\alpha}{k} (-1)^{k+1} \left(\psi(r)\right)^{k-1} r^{-d(k-1)}} \leq \varepsilon.\]
} It is thus a consequence of (i) and \equ{taylor} 
that $m(A_{f, \psi, \nu} ) = \infty$ if and only if  the integral \equ{correctone} is infinite; 
hence, statement (ii)  follows from Theorem \ref{mainapproxthmnew}. 
As for (iii), note that the series \equ{iii} diverges when $d > n$; thus, let us assume $d \leq n$. 
Using (i) and \equ{taylor}, we infer that there exists some $C \in \R_{>0}$ such that for any $T \in [S, \infty),$ we have 
\[
\begin{aligned}
m\left(B_{f,\psi(T),\nu,T}\right) &= m\left(\left\lbrace \xx \in \R^n : |f(\xx)| \le  {\psi(T)} \ \text{and} \  
 \nu(\xx) \leq T \right\rbrace\right) \\  
&\ge m\left(A_{f, \psi(T), \nu} \cap \{\xx \in \R^n : S \leq \nu(\xx) \leq T \} \right) \\  
&\ge C \psi(T) \int_{S}^{T} {z}^{n-(d+1)} \,d{z}  \\  
& = 
 \begin{cases}\ds C \psi(T)  (\log T - \log S) &\text{ if }d = n,\\ \ds (n-d)^{-1} C \psi(T)\left( T^{n-d} - S^{n-d}\right)&\text{ if }d < n.
\end{cases}
\end{aligned}
\]
Thus, there exists some $C' \in \R_{>0}$ such that for each sufficiently large $T \in \R_{>0}$, we have 
\[
m\left(B_{f,\psi(T),\nu,T}\right) \ge  
 \begin{cases}\ds C' \psi(T)  \log T  &\text{ if }d = n,\\ \ds C' \psi(T) T^{n-d} &\text{ if }d < n.
\end{cases}
\]
Letting $T = t_k = 2^k$ for each sufficiently large $k \in \Z_{\geq 1}$ and applying Theorem \ref{UniformMain} implies (iii). 
\end{proof}  

\begin{proof}[Proof of Corollary \ref{product}] 
{Recall that $\nu$ is the maximum norm on $\R^n$ and $f(\xx) = x_1\cdots x_n$.
For each $k \in \{{0}, \dots , n - 2 \}$ and $z \in \R_{>0},$ define 
\[ I_k(z,\psi) := \underbrace{\int_{0}^z \dots \int_0^z}_{k \ \mathrm{times}} \min\left( z, \frac{\psi(z)}{z\prod_{i=1}^{k} y_i}\right) \,dy_{1} \cdots dy_{k}; \] 
when $k=0$ here, there is no integration, and the empty product $\prod_{i=1}^{0} y_i$ is equal to $1$ by convention.
 It is easy to see that 
 the left hand side of  \equ{measureexact} is equal to
$\ds 2^n \, n \int_{S}^{T} I_{n-2}(z) \,dz$. 
For each $k$ and $z$ as above, we now establish the following explicit formula: 
 \eq{explicit}{ \begin{aligned}I_k(z,\psi) &= \min\left(z^{k+1},\frac{\psi(z)}{z} \right)  \sum_{i=0}^{k} \frac{1}{i!} \max\left(0,\log\Big(\frac{z^{k+2}}{\psi(z)}\Big)\right)^i  \\ &=  \begin{cases}\ds z^{k+1}&\text{ if }\psi(z)\ge z^{k+2};\\ \ds\frac{\psi(z)}{z}   \sum_{i=0}^{k} \frac{1}{i!} \log^i\Big(\frac{z^{k+2}}{\psi(z)}\Big)&\text{ otherwise.}\end{cases} \end{aligned}} (Here, the second equality is obvious.) In this formula and in the remainder of this proof, we adopt the convention $0^0 = 1.$ We now prove this formula by induction on $k.$ The base case $k=0$ is clear. Now suppose that for some $k \in \{1, \dots , n-3\},$ we have
\[
 I_{k-1}(z,\psi) = \min\left(z^{k},\frac{\psi(z)}{z} \right)  \sum_{i=0}^{k-1} \frac{1}{i!} \max\left(0,\log\Big(\frac{z^{k+1}}{\psi(z)}\Big)\right)^i;
\]
then
\eq{induction}{ \begin{aligned}I_k(z,\psi) &= \int_{0}^z  I_{k-1}\left(z,\tfrac1y\psi\right) \,dy \\ &= \int_{0}^z \min\left(z^{k},\frac{\psi(z)}{yz} \right)  \sum_{i=0}^{k-1} \frac{1}{i!} \max\left(0,\log\Big(\frac{yz^{k+1}}{\psi(z)}\Big)\right)^i \,dy.\end{aligned}}
Consider first the case $\psi(z)\ge z^{k+2}$. Then for any $0< y \le z,$ one has $$ z^k \le  \frac{\psi(z)}{yz} \ \Longleftrightarrow\ \frac{yz^{k+1}}{\psi(z)}\le 1;$$ and \equ{induction} gives 
$\ds I_k(z,\psi) = \int_{0}^z z^k \,dy = z^{k+1}$. If $\ds \psi(z) < z^{k+2},$ then
\begin{equation*}\begin{split}
I_k(z,\psi) &= \int_{0}^{\frac{\psi(z)}{z^{k+1}}} z^{k}  \,dy +  \int_{\frac{\psi(z)}{z^{k+1}}}^z \frac{\psi(z)}{yz}  \sum_{i=0}^{k-1} \frac{1}{i!} \log^i\Big(\frac{yz^{k+1}}{\psi(z)}\Big) \,dy\\
& = \frac{\psi(z)}{z}\left[ 1 +   \sum_{i=0}^{k-1} \frac{1}{i!}  \int_{\frac{\psi(z)}{z^{k+1}}}^z\log^i\Big(\frac{yz^{k+1}}{\psi(z)}\Big)\,\frac{dy}y\right] \\
& = \frac{\psi(z)}{z}\left[ 1 +   \sum_{i=0}^{k-1} \frac{1}{i!} \left.\frac1{i+1}\log^{i+1}\Big(\frac{yz^{k+1}}{\psi(z)}\Big)\right|^z_{\frac{\psi(z)}{z^{k+1}}}\right] \\
&= \frac{\psi(z)}{z}   \sum_{i=0}^{k} \frac{1}{i!} \log^i\Big(\frac{z^{k+2}}{\psi(z)}\Big).
\end{split}\end{equation*} 
This proves  \equ{explicit}. Now fix $M \in \R_{\geq 1}$ such that for each $z \in \R_{\geq M},$ we have $\psi(z) < z^n$. Then  \equ{explicit} implies that for each $z \in \R_{\geq M},$ we have $\ds I_{n-2}(z,\psi) = \frac{\psi(z)}{z}   \sum_{i=0}^{n-2} \frac{1}{i!} \log^i\Big(\frac{z^{n}}{\psi(z)}\Big)$. {This establishes \equ{measureexact} for any ${T}\ge {S} \ge M$ and} proves (i).
 }

{In the expression $\ds \sum_{i=0}^{n-2} \frac{1}{i!} \log^i\Big(\frac{z^{n}}{\psi(z)}\Big),$ the term that corresponds to $i = n-2$ dominates as $z\to\infty$}; statement (ii)  now follows from (i) and Theorem \ref{mainapproxthmnew}. 

\smallskip

{Finally, arguing as in the proof of Corollary \ref{quadforms}(iii), we see that 
there exist   $S\in \R_{\ge M}$ and $C,C' \in \R_{>0}$ such that for any $T \in [S, \infty)$ we have 
\[
\begin{aligned}
m\left(B_{f,\psi(T),\nu,T}\right) 
&\ge m\left(A_{f, \psi(T), \nu} \cap \{\xx \in \R^n : S \leq \nu(\xx) \leq T \} \right) \\  
&\ge C \int_{S}^{T} \frac{\psi({T})}{{z}}
  \log^{n-2}\Big(\frac{{z}^n}{\psi({T})}\Big)
   \,d{z}  \\  
&= n^{n-2} \, C \psi(T) \int_{S}^{T} 
  \log^{n-2}\Big(\frac{{z}}{\psi({T})^{1/n}}\Big)
   \,\frac{d{z}}z  \\  
& \ge 
 \begin{cases}\ds C' \psi(T)  (\log T - \log S) &\text{ if }n=2,\\ C' \psi(T) \left[  \log^{n-1}\Big(\frac{{T}}{\psi({T})^{1/n}}\Big) -  \log^{n-1}\left(\frac{{S}}{\psi({T})^{1/n}}\right)\right]&\text{ if }n > 2.
\end{cases}
\end{aligned}
\]
Thus, there exists   $C'' \in \R_{>0}$ such that for each sufficiently large $T \in \R_{>0}$, we have 
\[
m\left(B_{f,\psi(T),\nu,T}\right) \ge  
 \begin{cases}\ds C'' \psi(T)  \log T  &\text{ if }n=2,\\ \ds C'' \psi(T)  \log^{n-1}\Big(\frac{{T}}{\psi({T})^{1/n}}\Big) &\text{ if }n>2.
\end{cases}
\] 
Letting $T = t_k = 2^k$ for each sufficiently large $k \in \Z_{\geq 1}$ and applying Theorem \ref{UniformMain} implies (iii). }
\end{proof}

\begin{proof}[Proof of Corollary \ref{khintchine}] 
{Recall that  $\nu$ is the maximum norm on $\R^n,$ and $$f(x_1, \dots , x_n) = \max\left(|x_1|^{a_1}, \dots , |x_\ell|^{a_\ell}\right).$$}
Fix $M \in \R_{> 0}$ such that for each $z \in \R_{\geq M}$ and each $i \in \{1, \dots , \ell \}$ we have $\ds {\psi(z)^{{1}/{a_i}} < M}.$ 
{Take $T\ge S \ge M$, and} set $$  A := A_{f, \psi, \nu } \cap \{\xx \in \R^n : {S} \leq \nu(\xx) \leq {T} \}$$ and 
$\ds A_{\geq 0} := A \cap \left(\R_{\geq 0}\right)^n.$ By symmetry, it is clear that $\ds m(A) = 2^n \, m\left(A_{\geq 0}\right).$ {For any $\xx = (x_1, \dots , x_n) \in A_{\geq 0}$
and  any $i \in \{1, \dots , \ell \}$ one has 
$$|x_i| \leq \psi\big(\nu(\xx)\big)^{{1}/{a_i}} \leq \psi({S})^{{1}/{a_i}}< S \le \nu(\xx),$$
which implies $\nu(\xx) \in \left\lbrace |x_{\ell +1}|, \dots |x_n| \right\rbrace.$ 
Consequently,
$$A_{\geq 0} = \bigcup_{i = 1, \dots , \ell; \ j= \ell+1, \dots , n}B_{ij},$$
where for any $i \in \{1, \dots , \ell \}$ and  $j \in \{\ell+1, \dots , n\}$ we  set
\[ B_{i, j} := A_{\geq 0} \cap \left\lbrace \xx 
 \in \R^n : \max\left(|x_1|^{a_1}, \dots , |x_\ell|^{a_\ell}\right) = |x_i|^{a_i} \ \text{and} \ \nu(\xx) = |x_j| \right\rbrace.\]
In other words,
\[
B_{i, j} = \left\{\xx\in\R^n \left|\begin{aligned} &\quad{S} \leq x_j \leq {T}, \quad 0 \leq x_i \leq \psi\left(x_j\right)^{{1}/{a_i}}, \\  &\ 0 \leq x_p \leq x_i^{{a_i}/{a_p}}\ \forall\,p \in 
\{1, \dots , \ell\} \ssm \{i\}
, \\ &\ 0 \leq x_q \leq x_j \ \forall\, q \in 
\{\ell+1, \dots , n \} \ssm \{j\}
\end{aligned}\right.\right\}.
\] 
Therefore, 
\begin{equation*}
\begin{split}
m(B_{i, j}) &= \int_{S}^{T}  x_j^{n-\ell -1}   \int_0^{ \psi(x_j)^{{1}/{a_i}}} \prod_{p \in 
\{1, \dots , \ell\} \ssm \{i\}
}x_i^{{a_i}/{a_p}} \,dx_i \,dx_j  \\
\left(\text{recall the notation }a = \sum_{i=1}^\ell a_i^{-1}\right)\qquad\quad& = \int_{S}^{T}  x_j^{n-\ell -1}   \int_0^{ \psi(x_j)^{{1}/{a_i}}} x_i^{{a_i}(a-a_i^{-1})} \,dx_i \,dx_j  \\
& = \int_{S}^{T}  x_j^{n-\ell -1} \frac1{a_ia}  \big( \psi(x_j)^{{1}/{a_i}}\big)^{a_ia}  \,dx_j  \\
&= \frac1{a_ia} \int_{S}^{T}  z^{n-\ell -1}    \psi(z)^{a}  \,dz.  \\
\end{split}\end{equation*}
It follows that $$m(A) = 2^n (n-\ell) \sum_{i=1}^\ell \frac{1}{
  a_ia} \int_{S}^{T} \psi(z)^a z^{n-(\ell+1)} \ dz = 2^n (n-\ell) \int_{S}^{T} \psi(z)^a z^{n-(\ell+1)} \ dz,$$ 
which proves (i).} The other statements follow by arguing as in the proofs of {Corollaries \ref{quadforms} and \ref{product}}.  
\end{proof}  

\begin{rmk} \rm \label{kgt}
{More generally, one can take $\ell$ nonincreasing and regular functions $\ds \psi_1,\dots,\psi_\ell : \R_{\geq 0} \to \R_{>0}$ and, using the same argument, show that 
for a.e.\ $g\in G$, the system of inequalities
\[
|(g\vv)_i| \le \psi_i\big(\nu(\vv)\big),\quad i = 1,\dots,\ell
\]
(here, $(g\vv)_i$ denotes the $i^{\rm th}$ component of $g\vv$) has finitely (resp.\ infinitely) many solutions $\vv\in\mathcal{P}$ if and only if the integral  $$ \int_1^{\infty} \left[\prod_{i=1}^\ell\psi_i({z}) \right]{z}^{n - (\ell+1)} \,d{z}$$ is finite (resp.\ infinite). 
In the case $G = \SL$ and $\mathcal{P} = \Z^n_{\ne0}$ this can also be derived from Schmidt's generalization of the Khintchine--Groshev Theorem; see \cite[Theorem~2]{SchmidtKG}.}
\end{rmk}

\section{Concluding remarks}\label{cr}

\subsection{Inhomogeneous approximation}\label{inh} It is a natural problem to extend the methods of this paper to the inhomogeneous setting: that is, to study integer solutions of the system \equ{inhsystem} for fixed $\xi$ and almost every $g$, or vice versa, as is done in  \cite{GKY_a, GKY_b} for quadratic forms. Essentially this amounts to replacing the function $f$ with $f - \xi$, thereby getting rid of the subhomogeneity condition, which is crucial for transforming the results about generic lattices to those for generic forms.

On the other hand, it is not hard to see that the assumption of Theorem \ref{mainapproxthmnew} that $f 
$ be subhomogeneous may be replaced with the  assumption that $f$ be Borel measurable and 
\smallskip

\begin{itemize}
\item for any norms $\eta$ and $\nu$ on $\R^n$ and any $s \in \R_{>0},$ the Lebesgue measure of $A_{f, s\psi, \eta}$ is finite if and only if that of $A_{f, \psi, \nu}$ is finite.
\end{itemize}
\smallskip

 Similarly, one can weaken the subhomogeneity assumption of Theorem \ref{UniformMain}. This makes it possible to consider inhomogeneous problems for some classes of functions $f$, to be addressed in a forthcoming paper.

\subsection{Counting the number of solutions}\label{asymp} A comparison of Theorem \ref{genericcounting}(ii) with Theorem \ref{mainapproxthmnew}(ii) clearly shows a loss of information: a precise counting result for the number of lattice points in an increasing family of subsets of $\R^n$ turns into a rough estimate in the setting of generic subhomogeneous functions, with constants dependent on a compact subset of $G$. It is not clear whether, in the setting of Theorem \ref{mainapproxthmnew}(ii), the limit
$$
\lim_{T \to \infty}\frac{\card\left\lbrace \vv \in {{\mathcal{P}}} : (f \circ  g )(\vv) \les \psi\big( \nu(\vv) \big) \ \mathrm{and} \   \nu(\vv) \leq T \right\rbrace}{m\left(\left\lbrace \ttt \in \R^n : f(\ttt) \les   \psi\big(\nu(\ttt)\big) \ \mathrm{and} \   \nu(\ttt) \leq   T   \right\rbrace\right)} $$
exists for almost every \ $g\in G$. It is also not clear whether any  Khinchine-type results 
can be established without assuming the regularity of $\psi$.
 
\subsection{More metric number theory}\label{mnt} In general, the philosophy of this paper has been rooted  in metric Diophantine approximation, which, in its simplest incarnation, studies the rate of approximation of typical real numbers $\alpha$ by rational numbers $p/q$. 
Our Khintchine-type theorems naturally give rise to many further questions. For  example, in the $m\left(A_{f, \psi, \eta}\right) = \infty$ case of Theorem~\ref{simplemainapproxthm}, one may wish to study the \hd\ of the null set consisting of all $g \in \SL$ for which $f\circ g$ is not $(\psi, \nu)$-approximable. In the case of a critical exponent, this might produce an analogue of {badly approximable} objects that constitute a set that is of full Hausdorff dimension or is winning in the sense of W.\,M.~Schmidt  \cite{Sc1}. 
Such a result is established in \cite{KW} for binary indefinite quadratic forms, that is, for the case $n = d = 2$ of Corollary \ref{spequadforms}. 
Namely, let $\nu$ be an arbitrary norm on $\R^2$. Then it follows from \cite[Theorem 1.2]{KW} that the set 
\eq{ba}{\{g\in \SLtwo: \exists\,\varepsilon > 0 \text{ such that  $f\circ g$ is not 
$(\varepsilon, \nu)$-approximable}\}} has full Hausdorff dimension. (The proof actually yields a stronger  \textsl{hyperplane absolute winning} property introduced in \cite{BFKRW} and known to imply winning, which then implies full \hd.) See also \cite{AGK} for higher-dimensional generalizations.  Note that 
Corollary \ref{spequadforms} implies that  the set \equ{ba} has Haar measure zero.

Alternatively, in the $m\left(A_{f, \psi, \eta}\right) < \infty$ case of Theorem~\ref{simplemainapproxthm}, one can ask for the Hausdorff dimension of the null set consisting of all $g \in \SL$ for which $f\circ g$ is  $(\psi, \nu)$-approximable. It seems natural to seek an analogue of the mass transference principle of Beresnevich\textendash Velani in \cite{BV}; that being said, the $\limsup$ sets in question have  a complicated structure, and the standard techniques do not appear to be applicable.

\end{document}